\documentclass[reqno,11pt]{amsart}
 
\usepackage{amsmath}
\usepackage{amsthm}  
\usepackage{verbatim}
\usepackage{enumerate} 
\usepackage{mathtools}  
\usepackage{amssymb}
\usepackage{mathrsfs}
\usepackage[top=1.5in, bottom=1.5in, left=0.7in, right=0.7in]{geometry}  
\usepackage[colorlinks]{hyperref}
\hypersetup{
	colorlinks,
	citecolor=blue,
	linkcolor=red
}   
\numberwithin{equation}{section}
\usepackage{xypic}
\usepackage{bm}
\usepackage{tikz}
\usetikzlibrary{decorations.pathreplacing}
\usepackage{xcolor}
\usepackage{float}
\usepackage{cleveref}
\usepackage{comment}

\newtheorem{theorem}{\textbf{Theorem}}[section]
\newtheorem{theorem*}{\textbf{Theorem}}

\newtheorem{definition}[theorem]{\textbf{Definition}}
\newtheorem{proposition}[theorem]{\textbf{Proposition}}

\newtheorem{question}[theorem]{Question}

\newtheorem{corollary}[theorem]{\textbf{Corollary}}
\newtheorem{remark}[theorem]{\textbf{Remark}}

\newtheorem{definition/proposition}[theorem]{\textbf{Definition/Proposition}}

\def\N{{\mathbb N}}
\def\R{\mathbb{R}}
\def\Z{{\mathbb Z}}

\def\C{{\mathbb C}}
\def\D{{\mathbb D}}

\def\Q{{\mathbb Q}}

\def\cA{{\mathcal A}}

\def\cM{{\mathcal M}}

\def\cR{{\mathcal R}}

\def\rd{{\rm d}}

\def\la{\langle\,}
\def\ra{\,\rangle}

\def\std{\rm std}

\DeclareMathOperator{\Ima}{im}
\DeclareMathOperator{\ind}{ind}

\DeclareMathOperator{\Id}{id}

\DeclareMathOperator{\virdim}{virdim}

\DeclareMathOperator{\CHA}{CH}
\DeclareMathOperator{\CH}{CH}
\DeclareMathOperator{\CC}{CC}

\newcommand{\Addresses}{{
		\bigskip
		\footnotesize

	    Zhengyi Zhou, \par\nopagebreak
        \textsc{State Key Laboratory of Mathematical Sciences, Chinese Academy of Sciences;}\par\nopagebreak
	    \textsc{Morningside Center of Mathematics, Chinese Academy of Sciences;}\par\nopagebreak
         \textsc{Academy of Mathematics and Systems Science, Chinese Academy of Sciences, China}\par\nopagebreak
		\textit{E-mail address}: \href{mailto:zhyzhou@amss.ac.cn}{zhyzhou@amss.ac.cn}

}}

\title{Contact $(+1)$-surgeries and algebraic overtwistedness}
\author{Zhengyi Zhou}

\begin{document}
	\maketitle
\begin{abstract}
We show that a contact $(+1)$-surgery along a Legendrian sphere in a flexibly fillable contact manifold ($c_1=0$ if not subcritical) yields a contact manifold that is algebraically overtwisted if the Legendrian's homology class is not annihilated in the filling. Our construction can also be implemented in more general contact manifolds yielding algebraically overtwisted manifolds through $(+1)$-surgeries. This gives a new proof of the vanishing of contact homology for overtwisted contact manifolds. Our result can be viewed as the symplectic field theory analog in any dimension of the vanishing of contact Ozsv\'ath-Szab\'o invariant for $(+1)$-surgeries on two-component Legendrian links proved by Ding, Li, and Wu \cite{DLW}. 
\end{abstract}
%\tableofcontents
\section{Introduction}
The fundamental dichotomy in contact topology separates manifolds into the collection of overtwisted contact manifolds, which are flexible in the sense that an $h$-principle holds by the seminal work of Eliashberg \cite{zbMATH04121041} and Borman-Eliashberg-Murphy \cite{zbMATH06567662}, and the collection of tight contact manifolds, where some forms of symplectic rigidity are expected. Understanding the boundary between these two phenomena in various forms is at the heart of contact topology.

One way to study contact manifolds is from a surgical perspective. Weinstein \cite{zbMATH00011093} showed that one can modify a contact manifold by attaching a symplectic handle along a neighborhood of an isotropic sphere, which is now referred to as a Weinstein handle \cite{zbMATH06054083}. Such a procedure is called an isotropic surgery by Conway and Etnyre \cite{zbMATH07206659}. One can reverse the procedure by attaching a symplectic handle along a neighborhood of a coisotropic sphere, which leads to the concept of coisotropic surgeries \cite{zbMATH07206659}. Among them, arguably, the most interesting case is when the sphere is both isotropic and coisotropic, i.e.\ Legendrian. An isotropic surgery along a Legendrian sphere is often called a contact $(-1)$ surgery, while the coisotropic surgery along the Legendrian sphere is called a contact $(+1)$-surgery. Ding and Geiges \cite{zbMATH02103046} showed that every closed\footnote{All contact manifolds are assumed to be closed in this paper.} contact $3$-manifold can be obtained by contact $(\pm 1)$-surgery along a Legendrian link in the standard contact $3$-sphere. In higher dimensions, Conway and Etnyre \cite{zbMATH07206659} showed that any contact manifold can be obtained from the standard contact sphere via a sequence of isotropic and coisotropic surgeries. Therefore, to determine whether a contact manifold is overtwisted or tight, one needs to understand if tightness is preserved in surgeries. In dimension $3$, by the work of Colin \cite{MR1447038} and Wand \cite{zbMATH06487151}, isotropic surgeries preserve tightness. On the other hand, the contact $(+1)$ surgery along the standard unknot in the standard contact $3$-sphere yields a tight manifold, while we have an overtwisted manifold if we stabilize the unknot and apply the surgery. Hence the devil in the question is coisotropic surgeries, in particular, contact $(+1)$-surgeries.

Invariants from pseudo-holomorphic curves, e.g.\ symplectic field theory (SFT) by Eliashberg, Givental, and Hofer \cite{zbMATH01643843} and Heegaard Floer theory by Ozsv\'ath and Szab\'o  \cite{zbMATH02144173}, provide necessary conditions for a contact manifold to be overtwisted. Namely, the contact homology must vanish \cite{zbMATH05709738} from the SFT side and the contact Ozsv\'ath-Szab\'o invariant must vanish \cite{zbMATH02207895} from the Heegaard Floer theory side. Bourgeois and Niederkr{\"u}ger \cite{zbMATH05658836} introduced the notion of algebraically overtwisted manifolds for those with vanishing contact homology. However, neither condition is sufficient by Avdek \cite{avdek2020combinatorial} and Ghiggini, Honda, and Van Horn-Morris \cite{ghiggini2007vanishing}, hence the combination of the two vanishing properties does not imply overtwistedness via a contact connected sum either. From the surgical perspective, the non-vanishing of contact homology and the non-vanishing of the contact Ozsv\'ath-Szab\'o invariant (both of which hold for the standard contact sphere when the invariants are defined) are preserved in isotropic surgeries by the functoriality of those invariants. Their behavior under coisotropic surgeries is more complicated, as illustrated by the same example above. Even though both conditions are not sufficient to determine overtwistedness, understanding them in $(+1)$-surgeries can be viewed as the first step towards the geometric question of overtwistedness through coisotropic surgeries. On the Heegaard Floer theory side, a complete answer for the vanishing of the contact Ozsv\'ath-Szab\'o invariant in $(+1)$-surgery along a Legendrian knot was obtained by Golla \cite{zbMATH06413573}. In \cite{DLW}, Ding, Li, and Wu studied the vanishing of the contact Ozsv\'ath-Szab\'o invariant for $(+1)$-surgeries on two-component links. On the SFT side, the vanishing of contact homology through $(+1)$-surgeries was first studied by Avdek \cite{avdek2020combinatorial} in the standard contact $3$-sphere. In this paper, we study the same question but for general dimensions. In particular, our main theorem below can be viewed as an SFT analog of Ding-Li-Wu's result.
\begin{theorem}\label{thm:main}
    Let $Y^{2n-1}$ be the contact boundary of a Liouville domain $W$, where $W$ is one of the following:
    \begin{enumerate}
        \item $W=V\times \D$ for a Liouville domain $V$ and $\D\subset \C$ is the unit disk, in particular, any subcritical Weinstein domain.
        \item $W$ is a flexible Weinstein domain \cite{zbMATH06054083} with $c_1(W)\in H^2(W;\Z)$ torsion.
    \end{enumerate}
    Let  $\Lambda$ be a Legendrian sphere in $Y$, such that $[\Lambda]\in H_{n-1}(\partial W;\Q)$ is nontrivial in $H_{n-1}(W;\Q)$.   Then the contact manifold $Y_{\Lambda}$ from a $(+1)$-surgery\footnote{The $+1$ surgery depends on a parametrization $\Lambda\simeq S^{n-1}$.} along $\Lambda$ is algebraically overtwisted, i.e.\ the contact homology vanishes.
\end{theorem}
\begin{comment}
\begin{remark}\label{rmk:twist}
    Moreover, the contact homology over the twisted coefficient $\Q[H_2(Y_{\Lambda};\R)]$\footnote{It corresponds to $\cR=0$, i.e.\ the fully twisted theory in \cite{zbMATH06000009}.}  also vanishes for all contact manifolds from $(+1)$-surgeries in this paper. This implies all such contact manifolds have no weak fillings by \cite[Theorem 5]{zbMATH06000009}.
\end{remark}
\end{comment}
An immediate corollary of \Cref{thm:main} is the following.
\begin{corollary}\label{cor:OT}
    Let $V$ be a Liouville domain and $L\subset V$ be a Lagrangian sphere such that $[L]\ne 0 \in H_*(V;\Q)$, then for any Dehn-Seidel twist $\tau_L$\footnote{As a Dehn-Seidel twist also depends on a parametrization $L\simeq S^n$.}, the open book $\mathrm{OB}(V,\tau^{-1}_{L})$ with page $V$ and monodromy $\tau^{-1}_L$ has vanishing contact homology.
\end{corollary}
\begin{proof}
    The open book $\mathrm{OB}(V,\tau^{-1}_{L})$ is obtained from $(+1)$-surgery on the Legendrian lift of $L$ in the open book $\mathrm{OB}(V,\Id)=\partial(V\times \D)$, then \Cref{thm:main} applies as $H_*(V;\Q)\to H_*(\partial(V\times \D);\Q) \to H_*(V\times \D;\Q)$ is injective.
\end{proof}

In particular, the homotopically standard overtwisted $S^{2n+1}= \mathrm{OB}(DT^*S^n,\tau^{-1})$ has vanishing contact homology; this was established by Bourgeois and van Koert \cite{zbMATH05709738} by direct computation. The assumption on the fundamental class of $L$ is likely to be redundant in view of the regular Lagrangian conjecture of Eliashberg, Ganatra, and Lazarev \cite[Problem 2.5]{zbMATH07195660}. Although many of the open books in \Cref{cor:OT} are negative stabilization, hence overtwisted \cite{zbMATH07010365}, it is unclear whether \Cref{cor:OT} always yields overtwisted manifolds.

\begin{corollary}\label{cor:3D}
    Let $\Lambda \cup U$ be a two-component link in $(S^3,\xi_{\std})$ with a nontrivial linking number, where $U$ is the standard unknot. Then the $(+1)$ surgeries along $\Lambda \cup U$ yield a contact manifold with vanishing contact homology.
\end{corollary}
\begin{proof}
    We first apply $(+1)$ surgery along $U$ to get $Y=\partial(DT^*S^1\times \D)=S^1\times S^2$, then $\Lambda$ becomes a Legendrian knot $\Lambda'$ on $Y$ representing a nontrivial homology class in the $S^1$ factor as the linking number is nontrivial. Then we apply \Cref{thm:main} to $\Lambda'$.
\end{proof}

Ding, Li, and Wu \cite[Theorem 1.1]{DLW} showed that the contact Ozsv\'ath-Szab\'o invariant also vanishes for contact manifolds in \Cref{cor:3D}. In fact, they established the vanishing result for other types of $U$, which are ``unknots" in the Heegaard Floer theory sense. On the other hand, the nontrivial linking number is a crucial requirement, and so is the homology requirement in our formulation. Moreover, our construction enjoys a local property as follows.

\begin{theorem}\label{thm:main'}
    In the following two cases:
    \begin{enumerate}
        \item\label{thm1} $Y_1$ is flexibly fillable or $Y_1=\partial(V\times \D)$ such that $c_1(Y_1)$ is torsion, $Y_2$ is a contact manifold of the same dimension with $c_1(Y_2)$ torsion.
        \item\label{thm2} $Y_1=\partial(V\times \D)$ for a Weinstein domain $V$, $Y_2$ is a contact manifold of the same dimension.
    \end{enumerate}
    If $\Lambda$ is a Legendrian sphere in $Y=Y_1\# Y_2$, such that $[\Lambda]$ has non-trivial image through $H_{n-1}(Y;\Q)\simeq H_{n-1}(Y_1;\Q)\oplus H_{n-1}(Y_2;\Q)\to H_{n-1}(Y_1;\Q) \to H_{n-1}(W;\Q)$, where $W$ is the natural filling in \Cref{thm:main}, then $Y_{\Lambda}$ is algebraically overtwisted.
\end{theorem}
Then we can upgrade \Cref{cor:3D} to the following for the special case of $Y_1=\partial(DT^*S^{n-1}\times \D)$.
\begin{corollary}
    Let $\Lambda,U$ be two Legendrian spheres in $Y$, where $Y$ is a $2n-1$ dimensional contact manifold, and $U$ is a standard unknot in a Darboux chart. If the linking number is nontrivial\footnote{Here the linking number is defined to the intersection number of $\Lambda$ with a bounding ball of $U$ in the Darboux chart.},  the $(+1)$ surgeries along $\Lambda \cup U$ yield a contact manifold with vanishing contact homology.
\end{corollary}
As any overtwisted contact manifold $Y\# (S^{2n-1},\xi_{\mathrm{ot}})$ can be written as $(+1)$ surgeries from such links, this yields another proof that overtwisted contact manifolds have vanishing contact homology, which was first proved by Bourgeois and van Koert \cite{zbMATH05709738}. %Moreover, the contact homology with twisted coefficient also vanishes for overtwisted manifolds, which gives another proof of the following:

\begin{comment}
\begin{corollary}[\cite{zbMATH06182635,schmaltz2020non}]\label{cor:no_weak}
    Overtwisted contact manifolds have no weak filling.
\end{corollary}
\end{comment}

A $(+1)$-surgery gives rise to a Weinstein cobordism whose concave boundary is $Y_{\Lambda}$, while the convex boundary is $Y$ and $\Lambda$ is filled by a Lagrangian disk in the cobordism. On the other hand, the contact manifold $Y$ in \Cref{thm:main} enjoys a strong uniqueness property for symplectic fillings by \cite{zbMATH07367119,zbMATH07673358}, in particular, the homology class $[\Lambda]$ should survive in the filling. Indeed, one can apply \cite[Theorem 4.1]{bowden2022making} to prove that the contact manifold from $(+1)$-surgery has no strong fillings. The proof of \Cref{thm:main,thm:main'} follows from singling out the pseudo-holomorphic curves obstructing fillings, whose degeneration in the surgery cobordism then yields the vanishing of contact homology of the concave boundary. Contact homology of $(+1)$ surgeries was studied by Avdek \cite{Av,avdek2020combinatorial}, who explored a deeper picture relating the relative SFT of the convex boundary and the absolute SFT of the concave boundary. We point out here that our results remain in the realm of absolute SFT, i.e.\ we only use the topology of the surgery cobordism but not holomorphic curves with Lagrangian boundary conditions.

\begin{comment}
Our proof has a functorial explanation as follows. Given a contact manifold $Y$, one tries to define the positive symplectic cohomology, where the underlying cochain complex $C_+(Y)$ is generated by two generators for each Reeb orbit. $C_+(Y)$ does not always form a cochain complex, but $C_+(Y)\otimes \CC(Y)$ is a $\CC(Y)$ DGA-module, where $\CC(Y)$ is the contact homology algebra (chain level) of $Y$ and the differential counts Floer cylinders with negative punctures asymptotic to Reeb orbits. The cochain complex for the positive symplectic cohomology of an exact filling $W$ is then obtained by tensoring with the ground field using the augmentation induced by $W$. Now let $W$ be an exact cobordism (e.g.\ the surgery cobordism) from concave boundary $Y_-$ to convex boundary $Y_+$, then we have the following diagram, which is commutative on homology,
$$
\xymatrix{
C_+(Y_+)\otimes \CC(Y_+)\ar[d] \ar[r] & C(Y_+)\otimes \CC(Y_+)\ar[d]\\
C(W) \otimes \CC(Y_-)\ar[r] & C(Y_+)\otimes \CC(Y_-)}
$$
where $C(Y_{\pm}),C(W)$ are Morse cochain complexes. When phrased using such a structure, the core of the proofs is finding a closed class in $C_+(Y_+)\otimes \CC(Y_+)$ that is mapped to $\alpha \otimes 1 \in H^*(Y_+)\otimes \CH(Y_+)$ through the top map, such that $\alpha$ is not in the image of $H^*(W)\to H^*(Y_+)$. Then we must have $1=0$ in $\CH(Y_-)$. However, such an element is easy to find for $Y_+=Y$ in \Cref{thm:main,thm:main'}.
\end{comment}
It is quite a challenge to determine whether contact manifolds in \Cref{thm:main,thm:main'} are overtwisted. In dimension $3$, there are sufficient conditions for the $(+1)$-surgeries to yield overtwisted manifolds by Ozbagci \cite{zbMATH02147034} and Lisca-Stipsicz \cite{zbMATH05190395} for knots and Ding-Li-Wu \cite{DLW} for links. In higher dimensions, Casals, Murphy, and Presas \cite{zbMATH07010365} showed that $(+1)$-surgeries along loose Legendrian spheres give overtwisted manifolds. Indeed, some cases of \Cref{thm:main} give overtwisted manifolds, for example, $W=DT^*{S^{n-1}}\times \D$ and $\Lambda$ is the Legendrian lift of the Lagrangian zero section in $DT^*S^{n-1}$, as this Legendrian is loose/stabilized. On the other hand, there are Legendrian knots with the homology property in \Cref{thm:main} that are not stabilized found by Ekholm and Ng \cite[Corollary 2.22, Proposition 3.9]{zbMATH06471194}. In higher dimensions, we have many such Legendrians from exotic Weinstein balls constructed in \cite{abouzaid2010altering,zbMATH05553983,zbMATH02242665,zbMATH07367119} using the work of Lazarev \cite{zbMATH07305775}.
\begin{proposition}\label{prop:exotic}
    For $Y=\partial(DT^*S^{n-1}\times \D)\simeq S^{n-1}\times S^n$ with $n\ge 3$, there are infinitely many different non-loose Legendrian spheres in $Y$ that are smoothly isotopic to the standard loose $S^{n-1}$ above. When $n$ is odd\footnote{We expect this condition to be redundant.}, those Legendrians are formally Legendrian isotopic to the standard loose $S^{n-1}$.
\end{proposition}

One of the motivations of this project is to study the differences between overtwisted manifolds and algebraically overtwisted manifolds.
\begin{question}[Folklore]\label{question:AO}
For $n\ge 2$, are there algebraically overtwisted but tight $2n-1$ dimensional contact manifolds?
\end{question}
To put it in a broader perspective, this question is one of the fundamental questions to understand the boundary between flexibility and rigidity phenomena in symplectic and contact topology. The first example of an algebraically overtwisted tight manifold was found by Avdek \cite{avdek2020combinatorial} in dimension $3$. The example follows from a $(+1)$-surgery on $(S^3,\xi_{\std})$ along a trefoil knot, and the tightness follows from the non-vanishing of the contact Ozsv\'ath-Szab\'o invariant. In view of \cite[Theorem 1.1]{DLW}, although \Cref{thm:main} may give new examples in dimension $3$, the tightness criterion from the contact Ozsv\'ath-Szab\'o invariant does not apply, i.e.\ we need other criteria of tightness. In fact, the lack of tight criteria beyond contact homology is the major difficulty in answering \Cref{question:AO} in higher dimensions. Indeed, the existence of fillings, hypertight property, and properties on Conley-Zehnder indices, used as tight criteria in general dimensions, are all manifestations of the non-vanishing of contact homology. Nevertheless, \Cref{thm:main} potentially solves the easy half of \Cref{question:AO} by providing a flexible enough list of algebraically overtwisted manifolds. More precisely, we ask the following question.

\begin{question}
    If we apply a $(+1)$-surgery along Legendrian spheres in \Cref{prop:exotic}, do we get (different) tight contact manifolds?
\end{question}
We suspect the answer to be affirmative, for otherwise, the $(+1)$-surgeries would yield the homotopically standard overtwisted sphere, i.e.\ we get infinitely many different ways to get the homotopically standard overtwisted sphere but with the same formal data (at least for $n$ odd).
\subsection*{Acknowledgments}
The author is supported by the National Natural Science Foundation of China under Grant No.\ 12288201 and 12231010. The author is grateful to Russell Avdek for enlightening discussions that led to the functorial perspective in \S \ref{ss:43}, and Otto van Koert for pointing out \cite{zbMATH06562001} which leads to the proof of \Cref{prop:corb'} and for their feedback on a preliminary version of the paper. The author would like to thank Youlin Li and Zhongtao Wu for helpful discussions and interest in the project.
\section{Contact $(+1)$-surgeries}
In this section, we review contact $(+1)$-surgeries following Conway--Etnyre \cite{zbMATH07206659}. The definition of a contact $(+1)$-surgery along a Legendrian sphere is implicit in the theory of Weinstein handle attachments \cite{zbMATH06054083,zbMATH04147116,zbMATH00011093}. We first recall a model for the Weinstein $k$-handle for $k\le n$ as follows. Let $\omega = \sum_{i=1}^k\rd q_i\wedge \rd p_i + \sum_{i=1}^{n-k}\rd x_i\wedge \rd y_i$ be the standard symplectic structure on $\R^{2k}\times \R^{2n-2k}$. Then we have the following Liouville vector field 
$$v=\sum_{i=1}^k(-p_i\partial_{p_i}+2q_i\partial_{q_i})+\frac{1}{2}\sum_{i=1}^{n-k}(x_i\partial_{x_i}+y_i\partial_{y_i}).$$
For $a,b>0$, let $D_a$ be the disk of radius $a$ in the $p_i$ subspace and $D_b$ the disk of radius $b$ in the $q_i,x_i,y_i$ subspace. Then $(H_{a,b}=D_a\times D_b,\omega)$ is a model for the Weinstein $k$-handle. The Liouville vector field points outward along $\partial_+H_{a,b}:=D_a\times \partial D_b$ and inward along $\partial_-H_{a,b}:=\partial D_a \times D_b$, which in particular determines contact structures on the boundary.

Note that $S_a=\partial D_a \times \{0\}$ is an isotropic sphere and $S_b=\{0\} \times \partial D_b$ is a coisotropic sphere. By Moser's trick, the germ of the contact structure along $S_a\subset \partial_-H_{a,b}$ is contactomorphic to the germ of a contact structure along an isotropic sphere with a trivial conformal symplectic normal bundle (the quotient of the symplectic orthogonal of the tangent bundle by the tangent bundle) in any contact manifold. Similarly, the germ along $S_b\subset \partial_+H_{a,b}$ is contactomorphic to the germ along a coisotropic sphere in any contact manifold \cite[Lemma 3.4]{zbMATH07206659}. As a consequence, given an isotropic sphere with a trivialization of the conformal symplectic normal bundle in a contact manifold $Y$, we can glue $H_{a,b}$ to a neighborhood of the isotropic sphere along $\partial_-H_{a,b}$ for $b\ll 1$ using the Liouville vector field. This yields a Weinstein cobordism with concave boundary $Y$ and a new convex boundary $Y'$; we call this procedure an isotropic surgery. On the other hand, given a coisotropic sphere in $Y$, we can glue $H_{a,b}$ to a neighborhood of the coisotropic sphere along $\partial_+H_{a,b}$ for $a\ll 1$. This yields a Weinstein cobordism with convex boundary $Y$ and a new concave boundary $Y''$. This procedure is referred to as a coisotropic surgery. The isotropic surgery and coisotropic surgery are reverse operations to each other, i.e.\ we can undo an isotropic (resp.\ coisotropic) surgery by applying a coisotropic (resp.\ isotropic) surgery to the coisotropic (resp.\ isotropic) belt sphere in the surgery handle \cite[Lemma 3.9, Proposition 3.10]{zbMATH07206659}. 

By the functoriality of contact homology, only coisotropic surgeries have a chance to kill the contact homology. However, unlike the situation for isotropic spheres $S^{k-1}$ with $k<n$, coisotropic spheres do not enjoy an $h$-principle \cite[Theorem 1.3 and Remark 6.1]{zbMATH06670705}, which makes the existence of such manifolds much harder to identify. The situation for Legendrian spheres, i.e.\ spheres that are both isotropic and coisotropic, is a mixture of flexibility and rigidity: we have an $h$-principle for loose Legendrians by Murphy \cite{MR4172336}, yet there exists an enormous class of non-loose Legendrians exhibiting various forms of symplectic rigidity. 

\begin{definition}
A contact $(+1)$-surgery along a Legendrian sphere $\Lambda \subset Y$ is a coisotropic surgery along the Legendrian. We use $Y_{\Lambda}$ to denote the resulting contact manifold, and $W_{\Lambda}$ to denote the surgery Weinstein cobordism from $Y_{\Lambda}$ to $Y$.  
\end{definition}

\begin{remark}
    The $(+1)$-surgery depends on a parametrization $S^{n-1}\simeq \Lambda$. The resulting contact manifold, and even the underlying smooth manifold, may depend on this parametrization in general. We suppress this choice in our theorems, as the vanishing of contact homology does not depend on the parametrization.
\end{remark}

In \cite{zbMATH07455583}, Avdek gave an alternative definition of $(+1)$ surgery on Legendrian spheres, based on a Dehn--Seidel twist. This definition was used in \cite{zbMATH07010365}, where Casals, Murphy, and Presas show that the result of $(+1)$ surgery on a loose Legendrian knot is overtwisted.
\section{Contact homology and symplectic cohomology}
\subsection{Contact homology}
We will first recall the definition of contact homology. Let $(Y,\xi)$ be a co-oriented contact manifold and $\alpha$ a contact form such that all Reeb orbits are non-degenerate. Let $V$ denote the $\Q$ vector space generated by formal variables $q_\gamma$ for each good orbit $\gamma$ of $(Y,\alpha)$. We grade $q_\gamma$ by $|q_{\gamma}|:=\mu_{CZ}(\gamma)+n-3$ (the SFT degree), which should be understood as a well-defined $\Z/2$ grading in general unless $\det_{\C}\xi$ is trivialized (and a $\Q$-grading if $\det_{\C}\oplus^N \xi$ is trivialized for $N>1\in \N$). The chain complex  $\CC_*(Y,\alpha)$ for the contact homology is the free symmetric algebra $SV$. The differential is defined as follows.
\begin{equation}\label{eqn:partial}
\partial_{\CH}(q_{\gamma}) = \sum_{[\Gamma]} \#\overline{\cM}_{Y}(\{\gamma \},\Gamma) \frac{1}{\mu_{\Gamma}\kappa_{\Gamma}}q^{\Gamma}.
\end{equation}
The sum is over all multisets $[\Gamma]$, i.e.\ sets with duplicates. And $\Gamma$ is an ordered representation of $[\Gamma]$, e.g.\ $$\Gamma=\{\underbrace{\gamma_1,\ldots,\gamma_1}_{i_1}, \ldots, \underbrace{\gamma_m,\ldots,\gamma_m}_{i_m}\}$$ is an ordered set of good orbits with $\gamma_i\ne \gamma_j$ for $i\ne j$. We write $\mu_{\Gamma}=i_1!\ldots i_m!$ and $\kappa_{\Gamma}=\kappa^{i_1}_{\gamma_1}\ldots \kappa^{i_m}_{\gamma_m}$, which is the product of multiplicities of the Reeb orbits, and $q^{\Gamma}=q_{\gamma_1} \ldots  q_{\gamma_m}$. Here $\overline{\cM}_{Y}(\Gamma^+,\Gamma^-)$ is the SFT compactification \cite{zbMATH02062477} of the moduli space of rational holomorphic curves with asymptotic Reeb orbits $\Gamma^+=\{\gamma_1^+,\ldots,\gamma_{s^+}^+ \}$, $\Gamma^-=\{\gamma_1^-,\ldots,\gamma_{s^-}^- \}$ near positive and negative punctures respectively (with free asymptotic markers and prescribed parametrization of Reeb orbits, i.e.\ base points; the same applies to all curves with punctures asymptotic to Reeb orbits in the rest of the paper. This is the reason why the structural coefficient has $1/\kappa_{\Gamma}$), modulo the $\R$-translation, in the symplectization $\widehat{Y}=(\R_t\times Y, \rd(e^t\alpha))$ using an $\R$-invariant almost complex structure $J$ such that $J(\xi)=\xi$, $\rd(\cdot, J\cdot)$ is positive on $\xi$ and $J\partial_t = R$ (the Reeb vector field). The virtual dimension of $\overline{\cM}_{Y}(\Gamma^+,\Gamma^-)$ is given by
$$(n-3)(2-s^+-s^-)+\sum_{i=1}^{s^+}\mu_{CZ}(\gamma_i^+) - \sum_{i=1}^{s^-}\mu_{CZ}(\gamma_i^-)-1,$$
where the Conley-Zehnder index is defined using a symplectic trivialization of $u^*\xi$ (or $\det_{\C}u^*\xi$) for $u\in \overline{\cM}_{Y}(\Gamma^+,\Gamma^-)$.

The orientation property of $\overline{\cM}_{Y}(\{\gamma\},\Gamma)$ implies that \eqref{eqn:partial} is independent of the representative $\Gamma$. \eqref{eqn:partial} is always a finite sum by SFT compactness \cite[Theorem 10.1]{zbMATH02062477} as the Hofer energy is bounded by twice the period of $\gamma$ \cite[Lemma 5.16]{zbMATH02062477}. The differential on $\CC_*(Y,\alpha)$ is defined by the Leibniz rule
$$\partial(q_{\gamma_1} \ldots  q_{\gamma_l})=\sum_{j=1}^l (-1)^{|q_{\gamma_1}|+\ldots + |q_{\gamma_{j-1}}|}q_{\gamma_1} \ldots  q_{\gamma_{j-1}} \partial(q_{\gamma_j}) q_{\gamma_{j+1}}\ldots  q_{\gamma_l}.$$
The relation $\partial^2=0$ follows from the boundary configuration of $\overline{\cM}_{Y}(\{\gamma\},\Gamma)$ with virtual dimension $1$.

Given an exact cobordism $(X,\lambda)$ from $Y_-$ to $Y_+$, we have an algebra map $\phi$ from $\CC_*(Y_+,\lambda|_{Y_+})$ to $\CC_*(Y_-,\lambda|_{Y_-})$, which on generators is defined by
$$\phi(q_{\gamma})= \sum_{[\Gamma]}\# \overline{\cM}_{X}(\{\gamma\},\Gamma)\frac{1}{\mu_{\Gamma}\kappa_{\Gamma}}q^{\Gamma},$$
where $\Gamma$ is a collection of good orbits of $Y_-$. Here $\overline{\cM}_{X}(\Gamma^+,\Gamma^-)$ is the SFT compactification of the moduli space of rational holomorphic curves in the cobordism. The boundary configuration of $\overline{\cM}_{X}(\{\gamma\},\Gamma)$ with virtual dimension $1$ gives the relation $\partial\circ \phi = \phi \circ \partial$.

Even though contact homology is the simplest version of SFT, the analytical foundation to establish the counting of moduli spaces above as well as their algebraic relations is very involved and requires sophisticated virtual techniques going well beyond the classical approach of choosing $J$ carefully. The \emph{homology level} of contact homology was established by Pardon \cite{zbMATH07085531} using the implicit atlas and VFC \cite{zbMATH06578598} and by Bao--Honda \cite{zbMATH07656377} using Kuranishi perturbation theory. It is expected that they give the same contact homology. In this paper, we will use Pardon's construction and virtual techniques.

\begin{theorem}[\cite{zbMATH07085531}]
	The homology $\CHA_*(Y):=H_*(\CC_*(Y,\alpha))$ above realized in VFC is well defined and gives a monoidal functor from the symplectic cobordism category (with objects contact manifolds and morphisms exact cobordisms up to homotopy) to the category of (super)commutative algebras.
\end{theorem}

\begin{theorem}[\cite{zbMATH05709738}]\label{thm:AO}
    If $Y$ is overtwisted, then $\CHA_*(Y)=0$.
\end{theorem}

\begin{remark}
    The vanishing of contact homology implies the vanishing of rational SFT as well as SFT on the homology level \cite{zbMATH05658836}. Along with \Cref{thm:AO}, this led Bourgeois and Niederkr{\"u}ger to define contact manifolds with vanishing contact homology to be algebraically overtwisted manifolds. As overtwisted manifolds sit at the bottom of the symplectic cobordism category by the work of Etnyre--Honda \cite{zbMATH01801587} and Eliashberg--Murphy \cite{zbMATH07600533}, algebraically overtwisted manifolds sit at the bottom of the symplectic cobordism category in an algebraic sense, after which there are hierarchies describing the increasing levels of tightness using SFT; see Latschev--Wendl's algebraic torsions \cite{zbMATH06000009} and Moreno and the author's hierarchy functors \cite{moreno2020landscape}.
\end{remark}

\subsection{Symplectic cohomology}
Before reviewing symplectic cohomology, we first recall the following theorem from \cite{bowden2022making}, which serves as a motivation for \Cref{thm:main,thm:main'}.
\begin{theorem}[{\cite[Theorem 4.1]{bowden2022making}}]\label{thm:flex}
Let $W_0$ be a $(2n+2)$-dimensional flexible Weinstein domain, and let $(M,\xi)$ be any contact manifold.
    \begin{enumerate}
        \item[(a)]  Suppose that $c_1(W_0)$ and $c_1(\xi)$ both vanish. If $W$ is a strong filling of $Y=M\# \partial W_0$, then the kernel of $H_n(\partial W_0;\Q)\to H_n(W_0;\Q)$ contains the kernel of $$H_n(\partial W_0;\Q) \stackrel{0\oplus \Id }{\hookrightarrow}  H_n(M;\Q) \oplus H_n(\partial W_0;\Q)=H_n(M\# \partial W_0;\Q)\rightarrow H_n(W;\Q).$$ 
        \item[(b)] Let $W_0$ be subcritical (\textbf{without} any assumption on $c_1=0$). Then, for any strong filling $W$ of $Y$ we have that $H_n(\partial W_0;\Q)\to H_n(W;\Q)$ is injective.
        \end{enumerate}
\end{theorem}

As a corollary of \Cref{thm:flex}, we have the following weaker form of \Cref{thm:main,thm:main'}.
\begin{corollary}\label{cor:nofilling}
    Let $Y_{\Lambda}$ be the contact manifold in \Cref{thm:main,thm:main'} from $(+1)$-surgeries. Then $Y_{\Lambda}$ has no strong filling.
\end{corollary}
\begin{proof}
    If $Y_{\Lambda}$ has a strong filling, we then obtain a strong filling of the contact manifold in \Cref{thm:flex} by gluing with the surgery cobordism. Then the homology assumption on $\Lambda$ yields a contradiction with \Cref{thm:flex}. 
\end{proof}
The proof of \Cref{thm:flex} was based on the study of symplectic cohomology of strong fillings and was previously done in \cite{zbMATH07673358} for exact fillings. The strategy for \Cref{thm:main,thm:main'} is to extract the holomorphic curves (perturbed by a Hamiltonian) obstructing the filling in \Cref{cor:nofilling} and put them in the surgery cobordism; their degeneration will yield the vanishing of the contact homology of the concave boundary. 

\subsubsection{Basics of Symplectic cohomology}
Let $(W,\lambda)$ be an exact filling and $(\widehat{W},\widehat{\lambda})=W\cup (1,\infty)_r\times \partial W$ be the completion, where $\widehat{\lambda}=\lambda$ on $W$ and $\widehat{\lambda}=r\lambda|_{\partial W}$ on $(1,\infty)_r\times \partial W$. Note that $(0,\infty)_t\times \partial W = (1,\infty)_r\times \partial W$ for $r=e^t$. Let $H$ be a time-dependent Hamiltonian on the completion $(\widehat{W},\widehat{\lambda})$; the symplectic action for an orbit $\gamma$ is 
\begin{equation}\label{eqn:action}
\cA_H(\gamma)=-\int \gamma^*\widehat{\lambda}+\int_{S^1} (H\circ \gamma)\rd t.
\end{equation}
We say an almost complex structure $J$ is cylindrically convex near $\{r_0\} \times \partial W $ if and only if near the hypersurface $r=r_0$ we have that $\widehat{\lambda}\circ J =\rd r$ and $J$ is compatible with the symplectic form. We will consider a Hamiltonian $H$, which is a $C^2$ small perturbation (for the more precise meaning, see \eqref{ii} below) to the Hamiltonian that is $0$ on $W$ and linear with slope $a>0$ on $(1,\infty)_r\times \partial W$, such that $a$ is not a period of Reeb orbits of $R_{\lambda}$ for the contact form $\lambda$ on $\partial W$. We may assume the Hamiltonian is non-degenerate, where the Hamiltonian vector field is defined by $\rd \widehat{\lambda}(\cdot, X_{H})=\rd H$; then the periodic orbits consist of the following.
\begin{enumerate}[(i)]
	\item\label{i} Constant orbits on $W$ with $\cA_H\approx 0$.
	\item\label{ii} Non-constant orbits near Reeb orbits of $R_{\lambda}$ on $\{1\} \times \partial W$, with action close to the negative period of the Reeb orbits. In particular, we have $\cA_H\in (-a,0)$.  To see this, note that our $H$ is an $S^1$-dependent $C^2$ small perturbation to an autonomous Hamiltonian $h(r)$ with $h'(r)=a$ for $r>1+\epsilon$ with $\epsilon$ small and $h(r)=0$ for $r\le 1$. The non-constant orbits of $h(r)$ are in $S^1$ families like $(\gamma(h'(r_0)t),r_0)$, where $\gamma$ is a Reeb orbit of $R_{\lambda}$ with the period of $\gamma$ being $h'(r_0)$ for $1<r_0<1+\epsilon$. Therefore the symplectic action of such an orbit is
	$$-h'(r_0)r_0+h(r_0).$$
	It is clear that when $\epsilon\ll 1$, we have that the symplectic action is approximately the negative period of $\gamma$, which, in particular, is in $(-a,0)$. Then the $C^2$-small perturbation to $h(r)$ in \cite[Lemma 3.3]{MR2475400} will break the $S^1$ family orbits into two non-degenerate orbits with symplectic action arbitrarily close to the original $S^1$ family. We use $\check{\gamma},\hat{\gamma}$ to denote the two Hamiltonian orbits from a Reeb orbit $\gamma$; their difference is characterized by $\mu_{CZ}(\check{\gamma})=\mu_{CZ}(\gamma)$ and $\mu_{CZ}(\hat{\gamma})=\mu_{CZ}(\gamma)+1$. 
\end{enumerate} 
After fixing an $S^1$-dependent compatible almost complex structure $J$ that is cylindrically convex near a slice (i.e.\ a hypersurface $r=r_0$) where the Hamiltonian is linear with slope $a$, we can consider the compactified moduli space of Floer cylinders, i.e.\ solutions to $\partial_s u+J(\partial_t u-X_{H})=0$ modulo the $\R$ translation and asymptotic to two Hamiltonian orbits
$$\overline{\cM}_{H}(x,y)=\overline{\left\{u:\R_s\times S^1_t\to \widehat{W}\left|\partial_s u+J(\partial_t u-X_{H})=0, \lim_{s\to \infty} u(s,\cdot)=x, \lim_{s\to -\infty} u(s,\cdot)=y \right.\right\}/\R}.$$
With a generic choice of $J$, the count of rigid Floer cylinders defines a cochain complex $C^*(H)$, which is a $\Q$-vector space generated by Hamiltonian orbits. The differential $\delta$ is defined as
$$\delta(x)=\sum_{y,\dim \cM_{x,y}=0} \# \overline{\cM}_H(x,y)\cdot y.$$
The orbits of type \eqref{i} form a subcomplex $C^*_0(H)$, whose cohomology is $H^*(W;\Q)$. The orbits of type \eqref{ii} form a quotient complex $C^*_+(H)$. The cochain complexes are graded by $n-\mu_{CZ}$, which is in general a $\Z/2$ grading unless we choose a trivialization of $\det_{\C}TW$, and a $\Q$-grading if we choose a trivialization of $\det_{\C}\oplus^N TW$ for some $N\in \N_+$. Given two Hamiltonians $H_a,H_b$ with slopes $a<b$, we can consider a non-increasing homotopy of Hamiltonians $H_s$ from $H_b$ to $H_a$, i.e.\ $H_s=H_b$ for $s\ll0$ and $H_s=H_a$ for $s\gg 0$. Then the count of rigid solutions to the parameterized Floer's equation $\partial_s u+J(\partial_t u-X_{H_s})=0$ defines a continuation map $C^*(H_a)\to C^*(H_b)$, which is compatible with splitting into zero and positive complexes. Then the (positive) symplectic cohomology of $W$ is defined as
$$SH^*(W):=\lim_{a\to \infty} H^*(C^*(H_a)),\quad SH_+^*(W):=\lim_{a\to \infty} H^*(C^*_+(H_a)),$$
which fit into a tautological exact sequence,
$$\ldots \to H^*(W)\to SH^*(W)\to SH^*_+(W)\to H^{*+1}(W)\to \ldots $$
We define the filtered symplectic cohomology $SH^*_{<a}(W)$ by $H^*(C^*(H_a))$ and $SH^*_{+,<a}(W)$ by $H^*(C^*_+(H_a))$, which are independent of $H_a$ as long as $a$ is not a period of Reeb orbits on $\partial W$, see \cite[Proposition 2.8]{zhou2020mathbb}. The filtered symplectic cohomology depends on the Liouville form on the boundary, but we will suppress the notation of dependence as the specific contact form will be clear in the proofs.

One can also use autonomous Hamiltonians for symplectic cohomology, where Hamiltonian orbits come in $S^1$ families. The setup for symplectic cohomology requires choosing auxiliary Morse functions on the image of the Hamiltonian orbits to build the cascades model \cite{MR2475400}; see also \cite[\S 2.2]{zbMATH07673358}. Each Reeb orbit $\gamma$ gives rise to two generators of the cochain complex again, also denoted by $\check{\gamma},\hat{\gamma}$ corresponding to the minimum and maximum of the auxiliary Morse function on $S^1$. We use $\overline{\gamma}$ to represent an unspecified (check or hat) generator from a Reeb orbit $\gamma$.

We define 
$$\delta_{\partial}:SH_+^{*}(W)\to H^{*+1}(W;\Q)\to H^{*+1}(\partial W;\Q).$$
For any closed submanifold $S\subset \partial W$, we can define a map $C_+(H)\to \Q$ by counting the compactified moduli space  $\overline{\cM}_H(x,S)$
$$\overline{\left\{ u:\C\to \widehat{W}\left|\partial_su+J(\partial_tu-X_H)=0, \lim_{s\to \infty} u = x, u(0)\in \{1-\eta\} \times S \subset W  \right.\right\}/\R}$$
for $0<\eta \ll 1$. Here $H$ is assumed to be zero below the level of $r=1-\eta$. Such a map is clearly a cochain map, and, on the cohomology level, it is the same as the map $\la \delta_{\partial}(\cdot),[S]\ra:SH_+^*(W)\to \Q$, where $\la \cdot, \cdot\ra$ is the Kronecker pairing $H^*(W;\Q)\otimes H_*(W;\Q)\to \Q$. 

\subsubsection{Viterbo transfer}
Let $(V,\lambda_V)\subset (W,\lambda_W)$ be an exact subdomain, i.e.\ $\lambda_W|_V=\lambda_V$, then there exists $\epsilon>0$, such that $V_{\epsilon}:=V\cup  (1,1+\epsilon]_r\times \partial V \subset W$ with $\lambda_W|_{V_{\epsilon}}=\widehat{\lambda}_V$. Then we can consider a Hamiltonian $H_{VT}$ on $\widehat{W}$ as a $C^2$-small non-degenerate perturbation to the following.
\begin{enumerate}
	\item $H_{VT}$ is $0$ on $V$.
	\item $H_{VT}$ is linear with slope $B$ on $[1,1+\epsilon]_r\times \partial V$, such that $B$ is not the period of a Reeb orbit on $\partial V$.
	\item $H_{VT}$ is $B\epsilon$ on $W\backslash V_{\epsilon}$.
	\item $H_{VT}$ is linear with slope $A\le B\epsilon$ on $[1,\infty)_r\times \partial W$, such that $A$ is not the period of a Reeb orbit on $\partial W$.  
\end{enumerate}
Then there are five classes of periodic orbits of $X_{H_{VT}}$.
\begin{enumerate}[(I)]
	\item\label{I} Constant orbits on $V$ with $\cA_{H_{VT}}\approx 0$.
	\item\label{II} Non-constant periodic orbits near $\partial V$ with $\cA_{H_{VT}} \in (-B,0)$.
	\item\label{III} Non-constant periodic orbits near $\{1+\epsilon\}\times \partial V$ with action $\cA_{H_{VT}}\in (B\epsilon-(1+\epsilon)B, B\epsilon)=(-B,B\epsilon)$. To see the action region, it is similar to \eqref{ii} after \eqref{eqn:action}. Those orbits are close to Reeb orbits of $(\partial W, \lambda|_{\partial W})$, but placed near $r=1+\epsilon$. Therefore $-\int \gamma^*\widehat{\lambda}$ in \eqref{eqn:action} is close to $(1+\epsilon)$  times (as $\widehat{\lambda}|_{r=1+\epsilon}=(1+\epsilon)\lambda|_{\partial W}$) the period of the Reeb orbits and $\int_{S^1} (H_{VT}\circ \gamma) \rd t$ is close to $B\epsilon$. Hence the claim follows.
	\item\label{IV} Constant orbits on $W\backslash V_{\epsilon}$ with $\cA_{H_{VT}}\approx B\epsilon$.
	\item\label{V} Non-constant periodic orbits near $\partial W$ with $\cA_{H_{VT}}\in (B\epsilon-A, B\epsilon)$
\end{enumerate}
In particular, when $B\epsilon \ge A$, the quotient complex generated by orbits with non-positive action is generated by type \eqref{I}, \eqref{II} orbits along with some of the type \eqref{III} orbits. However, there is no Floer cylinder from a type \eqref{III} orbit to a type \eqref{I} or \eqref{II} orbit \cite[Figure 6]{cieliebak2018symplectic}. This follows from the asymptotic behavior lemma \cite[Lemma 2.3]{cieliebak2018symplectic} and the integrated maximum principle \cite[Lemma 2.2]{cieliebak2018symplectic}. Therefore orbits of the form \eqref{I}, \eqref{II} form a quotient complex when $B\epsilon\ge A$. Let $H_V$ denote the Hamiltonian on $\widehat{V}$ which is the linear extension of the truncation of $H_{VT}$ on $V_{\epsilon}$. Then by \cite[Lemma 2.2]{cieliebak2018symplectic}, the quotient complex is identified with $C^*(H_V)$. Those two applications of the integrated maximum principle are where the exactness of the cobordism $W\backslash V$ is crucial. Next, we consider a Hamiltonian $H_W$ on $\widehat{W}$ which is a $C^2$ small perturbation to the function that is zero on $W$ and linear with slope $A$ on $ (1,\infty)\times \partial W$. Then $H_W\le H_{VT}$ and we can find non-increasing homotopy from $H_{VT}$ to $H_W$, which defines a continuation map. Therefore we have a map 
$$C^*(H_W)\stackrel{\text{continuation}}{\longrightarrow} C^*(H_{VT}) \stackrel{\text{quotient}}{\longrightarrow} C^*(H_V).$$
Since the continuation map increases the symplectic action, the above map respects the splitting into $C_0,C_+$. Taking the direct limit for $B$ yields the Viterbo transfer map which is compatible with the tautological exact sequence,
$$
\xymatrix{\ldots \ar[r] & H^*(W;\Q) \ar[r] \ar[d] & SH^*_{<A}(W) \ar[r] \ar[d] & SH_{+,<A}^*(W) \ar[r]\ar[d] & H^{*+1}(W;\Q) \ar[r]\ar[d] & \ldots \\
	\ldots \ar[r] & H^*(V;\Q) \ar[r] & SH^*_{<B}(V) \ar[r] & SH_{+,<B}^*(V) \ar[r] & H^{*+1}(V;\Q) \ar[r] &  \ldots
 }
$$
If we also take the direct limit for $B$ first then for $A$, we get the Viterbo transfer map for the full symplectic cohomology.

\subsubsection{Symplectic cohomology of contact manifolds with DGA augmentations}\label{ss:aug}
Let $(W,\lambda)$ be an exact filling of $(Y,\alpha=\lambda|_Y)$, i.e.\ an exact cobordism from $\emptyset$ to $Y$, then the functorial package of contact homology gives rise to a differential graded algebra (DGA) morphism
$$\epsilon_W:\CC_*(Y,\alpha)\to \Q,$$
i.e.\ a DGA augmentation. On the other hand, given a DGA augmentation $\epsilon$, one can define the positive symplectic cohomology of $\epsilon$ as follows. Let $H$ be a Hamiltonian on $\widehat{Y}=\R_t\times Y \simeq (\R_+)_r\times Y$, where $r=e^t$, such that $H$ is $0$ on $(0,1]_r\times Y$ and is the same as the Hamiltonian on the cylindrical end (with slope $a$) in the definition of symplectic cohomology. Then we can consider the following compactified moduli space $\overline{\cM}_{Y,H}(x,y,\Gamma)$ for $\Gamma=\{\gamma_1,\ldots,\gamma_k\}$ a multiset of good Reeb orbits and $x,y$ \emph{non-constant} Hamiltonian orbits,
$$\overline{\left\{u:\R_s\times S^1_t\backslash \{p_1,\ldots,p_k\} \to \widehat{Y}\left|\partial_s u+J(\partial_t u-X_{H})=0, \lim_{s\to \infty} u(s,\cdot)=x, \lim_{s\to -\infty} u(s,\cdot)=y, \lim_{p_i} u= \gamma_i \right.\right\}/\R}.$$
Here $\lim_{p_i} u= \gamma_i$ is a short-hand for $u$ being asymptotic to $\gamma_i$ at $p_i$ viewed as a negative puncture with a free asymptotic marker, where the Hamiltonian is zero and the equation is the usual Cauchy-Riemann equation. Such compactification will be a mixture of Floer-type breaking at non-constant Hamiltonian orbits (it cannot break at a constant orbit of $H$ by symplectic action reasons) and SFT building breaking at the lower level. Then we define a differential $\delta_{\epsilon}$ on $C_+(H)$ as follows:
\begin{equation}\label{eqn:diff}
    \delta_{\epsilon}(x) = \sum_{[\Gamma]}  \frac{1}{\mu_{\Gamma}\kappa_{\Gamma}}\#\overline{\cM}_{Y,H}(x,y,\Gamma)\prod_{\gamma\in \Gamma}\epsilon(q_{\gamma}) \cdot y.
\end{equation}
The boundary configuration of $\overline{\cM}_{Y,H}(x,y,\Gamma)$ of virtual dimension $1$ gives that $\delta_{\epsilon}^2=0$. Similarly, by considering the compactified moduli space $\overline{\cM}_{Y,H}(x,C,\Gamma)$ for a singular chain $C\subset Y$ as follows, 
$$\overline{\left\{u:\C \backslash \{p_1,\ldots,p_k\} \to \widehat{Y}=\R_+\times Y\left|
\begin{array}{c}
\partial_s u+J(\partial_t u-X_{H})=0, p_i\ne 0,\\
\displaystyle \lim_{s\to \infty} u(s,\cdot)=x, u(0)\in \{ 1-\eta\}\times C, \lim_{p_i} u= \gamma_i
\end{array}\right.\right\}/\R},$$
for a fixed $0<\eta\ll 1$, we define $\delta_{\partial,\epsilon}:C^*_+(H)\to C^{*+1}(Y)$ by 
$$\delta_{\partial,\epsilon}(x)(C) =  \sum_{[\Gamma]} \frac{1}{\mu_{\Gamma}\kappa_{\Gamma}}\#\overline{\cM}_{Y,H}(x,C,\Gamma)\prod_{\gamma\in \Gamma}\epsilon(q_{\gamma}).$$
This is a cochain map. In actual constructions, to avoid working with the infinite-dimensional model $C^*(Y)$ (to avoid shrinking the space of admissible choices in the chosen virtual technique too much), one can use a finite-dimensional model for $Y$, for example, a simplicial complex from a triangulation or a Morse complex using an auxiliary Morse function. Moreover, if one only wants to understand the effect of $\delta_{\partial,\epsilon}$ on a homology class of $Y$ under the pairing $H^*(Y;\Q)\otimes H_*(Y;\Q)\to \Q$, one may take a singular chain or a submanifold representing the homology class. 

We use $SH^*_+(Y,\alpha,\epsilon)$ to denote the direct limit of the cohomology of $(C^*_+(H),\delta_{\epsilon})$ through the continuation map of increasing slopes (corrected by the augmentation in the same way as \eqref{eqn:diff}), and $\delta_{\partial,\epsilon}$ to denote the map $SH^*_+(Y,\alpha,\epsilon)\to H^{*+1}(Y;\Q)$. We summarize the status of such objects and their properties in the following.
\begin{enumerate}
    \item\label{1} The analytical foundation of such symplectic cohomology is established by Pardon's work on Hamiltonian-Floer cohomology \cite{zbMATH06578598} and contact homology \cite{zbMATH07085531}. But strictly speaking, as $\CC_*(Y,\alpha)$ as well as the counting of $\overline{\cM}_H(x,y,\Gamma)$ for $\Gamma=\{\gamma_1,\ldots,\gamma_k\}$ depend on auxiliary choices, i.e.\ choices in \cite[\S 4.8]{zbMATH07085531} (in addition to the choice of contact forms, and almost complex structures), the symplectic cohomology also depends on such choices a priori. That is the most rigorous notation should be 
    $$SH_+^*(Y,\alpha,J,\theta,\epsilon)$$
    where 
    \begin{enumerate}
        \item $\theta$ is a choice of auxiliary data from a set analogous to $\Theta$ in \cite[\S 4.8]{zbMATH07085531}. Moreover, $\theta$ extends the auxiliary choice $\theta_0$ needed in \cite[\S 4.8]{zbMATH07085531} to define the contact DGA $\CC_*(Y,\alpha)$ using the contact form $\alpha$ and the $\R$-invariant almost complex structure that is the same as $J$ on the negative end of $\widehat{Y}$ used for $\overline{\cM}_{Y,H}(x,y,\Gamma)$.
        \item $\epsilon$ is an augmentation of $\CC_*(Y,\alpha)$, which depends on $\alpha$, the almost complex structure as well as $\theta_0$.
    \end{enumerate}
    An alternative perspective is viewing $C^{-*}_+(H)\otimes \CC_*(Y,\alpha)$ as a DGA module over $\CC_*(Y,\alpha)$ with differential from counting $\overline{\cM}_{Y,H}(x,y,\Gamma)$ and the Leibniz rule. Then the symplectic cohomology is the tensor product with $\Q$, if we view $\Q$ as a $\CC_*(Y,\alpha)$ module from the augmentation $\epsilon$, see \S \ref{ss:43} for more details on this perspective.

    It is expected that such symplectic cohomology is independent of contact forms, almost complex structures, $\theta$, as well as augmentations up to DGA homotopies. As a consequence of this expectation, the set of all such symplectic cohomologies is a contact invariant. The usual invariance proof for symplectic cohomology/Hamiltonian-Floer cohomology can only show the invariance for extensions $\theta$ of $\theta_0$ and choices of Hamiltonians. Different choices of $\alpha,J$ and $\theta_0$ were shown to give rise to contact DGAs that are chain homotopic to each other \cite{zbMATH07085531}. On the other hand, the symplectic cohomology here is similar to the linearized contact homology. In particular, one would like to have DGA homotopies between them, which would allow us to define the homotopy classes of augmentations (independent of $\alpha,J,\theta_0$) and prove the invariance of $SH_+^*(Y,\alpha,J,\theta,\epsilon)$.
    \item When $\epsilon$ is $\epsilon_{W}$ from an exact filling (or a strong filling and we use the Novikov field coefficient), $SH^*_+(Y,\alpha,\epsilon_W)\to H^{*+1}(Y)$ is isomorphic to $SH^*_+(W)\to H^{*+1}(W)\to H^{*+1}(Y)$ for any choice of $\alpha$ and auxiliary choices in the VFC. The proof is no different from the functoriality of compositions for contact homology by neck-stretching in \cite[\S 1.5]{zbMATH07085531}. This perspective of symplectic cohomology was introduced by Bourgeois and Oancea \cite{MR2471597}, where they proved the equivalence between positive $S^1$-equivariant symplectic cohomology and linearized contact homology w.r.t.\ the augmentation from the filling. 
    \item Given an exact cobordism $(X,\lambda)$ from $Y_-$ to $Y_+$, we can consider a Hamiltonian $H_+$ on $\widehat{X}$ which is linear on the positive cylindrical end of $\widehat{X}$ and zero everywhere else. Given an augmentation $\epsilon$ of $Y_-$, one can define similarly the symplectic cohomology $SH_+^*(X,\lambda,\epsilon)$ with additional choices in the VFC as in \eqref{1} above. Similar to \eqref{1}, this cohomology depends a priori on the contact form, almost complex structure, as well as auxiliary data $\theta_0$ for the concave boundary.  $SH_+^*(Y,\alpha,\epsilon)$ may be viewed as the $SH_+^*(Y\times[0,1]_t,e^t\alpha,\epsilon)$.
    
    We consider a Hamiltonian $H_{VT}$ on $\widehat{X}$ that is similar to the Hamiltonian in the construction of the Viterbo transfer map; following the same recipe as the Viterbo transfer map, along with the correction from the augmentation $\epsilon$, we get a Viterbo transfer map
    $$SH_+^{*}(X,\lambda,\epsilon)\to SH_+^*(Y_-,\lambda|_{Y_-},\epsilon).$$
    By neck-stretching along $Y_+ \subset \widehat{X}$, we get an isomorphism 
    $$SH_+^*(X,\lambda,\epsilon)\simeq SH_+^*(Y_+,\lambda|_{Y_+},\epsilon\circ \phi_X),$$
    where $\phi_X$ is the DGA morphism from the cobordism $X$; hence $\epsilon\circ \phi_X$ is a DGA augmentation of $\CC(Y_+,\lambda|_{Y_+})$. 
    \item Similar to $\delta_{\partial}:SH^*_+(Y,\alpha,\epsilon)\to H^{*+1}(Y)$, we can define a map $SH^*_+(W,\lambda,\epsilon)\to H^{*+1}(W)$. Moreover, there exist auxiliary choices making the following diagram (Viterbo transfer) commute: $$
    \xymatrix{
    SH_{+}^*(Y_+,\alpha_+,\epsilon\circ\phi_{W})\ar[d]^{\simeq}  & \\
SH_{+}^*(W,\lambda,\epsilon)\ar[r]^{\phi_{Viterbo} }\ar[d] & SH_{+}^*(Y_-, \alpha_-, \epsilon_)\ar[d]^{\delta_{\partial}}\\
H^{*+1}(W) \ar[r] & H^{*+1}(Y_-)
}
$$
where $\alpha_{\pm}=\lambda|_{Y_{\pm}}$.
\end{enumerate}
We do not use the well-definedness (independence of various choices) of such cohomology in this paper, but rather the existence of one construction guaranteed by Pardon's VFC. The idea of the above Viterbo transfer maps with augmentations is used in the contact connected sum for \Cref{thm:main'}, but only in a very special form, where the augmentation is trivial. We also only require that there exist auxiliary choices making one construction of the Viberbo transfer. The key feature relating those abstract existences with geometric applications is the fact that the virtual count of a compactified moduli space equals the geometric count if the moduli space is cut out transversely, e.g.\ \cite[Proposition 4.33]{zbMATH07085531}.

\begin{remark}
One can modify the proof of \Cref{thm:flex} using symplectic cohomology of augmentations to prove that $Y_{\Lambda}$ in \Cref{thm:main,thm:main'} has no DGA augmentations in a similar way to \Cref{cor:nofilling}. 
\end{remark}
\section{Vanishing of contact homology}
In \S \ref{ss:41}, we establish the existence of curves hitting $\Lambda$ for contact manifolds in \Cref{thm:main} and prove \Cref{thm:main}. In \S \ref{ss:42}, we recall properties of the connected sum and prove \Cref{thm:main'}. In \S \ref{ss:43}, we give an explanation of our results from the functorial aspects of SFT. Finally in \S \ref{ss:44}, we explain \Cref{prop:exotic}.

\subsection{Holomorphic curves hitting $\Lambda$}\label{ss:41}
\subsubsection{Case 1: $Y=\partial(V\times \D)$} \label{ss:411}
 Let $f$ be a Morse function on $V$ such that $\partial_r f>0$ near the boundary, where $r$ is the collar coordinate. Following \cite[\S 2.1]{zbMATH07673358}, we can endow $Y$ with a contact form such that:
\begin{enumerate}
	\item Each critical point $p$ of $f$ corresponds to a simple non-degenerate Reeb orbit $\gamma_p$, which is the circle over $p$ in the region $V\times S^1\subset \partial(V\times \D)$;
	\item Using the natural bounding disk from the $\D$-component, the Conley-Zehnder index of $\gamma_p$ is given by $2+\dim_{\C}{V}-\ind(p)$, hence the SFT degree is positive;
	\item The period of $\gamma_p$ is approximately $1$, and $\gamma_p$ has longer period than $\gamma_q$ if and only if $f(p)<f(q)$. 
\end{enumerate}	
The orbits $\{\gamma_p\}$ from critical points are all the Reeb orbits with period at most $2$. We may assume orbits of period at most $8$ are covers of $\gamma_p$, whose SFT degrees are positive if $c_1(V)$ is torsion. If we use a Hamiltonian $H$ with slope $2$ that vanishes on the filling, then we have two non-constant Hamiltonian orbits $\check{\gamma}_p,\hat{\gamma}_p$ from each Reeb orbit $\gamma_p$, where $\mu_{CZ}(\check{\gamma}_p)=\mu_{CZ}(\gamma_p)$, $\mu_{CZ}(\hat{\gamma}_p)=\mu_{CZ}(\gamma_p)+1$. Moreover, we have the following properties \cite[Proposition 2.6]{zbMATH07673358}.
\begin{enumerate}
\item We have $SH^{*}_{+,<2}(V\times \D) \simeq H^*(V;\Q)[1]\oplus H^*(V;\Q)[2]$, where the $H^*(V;\Q)[1]$ component is generated by $\check{\gamma}_p$ and the $H^*(V;\Q)[2]$ component is generated by $\hat{\gamma}_p$.
\item $\delta_{\partial}$ restricted to $H^*(V;\Q)[1]$ is injective and is an isomorphism when restricted to $V\times \{\mathrm{pt}\}\subset V\times S^1\subset  \partial(V\times \D)$. $\delta_{\partial}$ is zero on the $H^*(V;\Q)[2]$ component. 
\end{enumerate}	
In particular, if $[\Lambda]$ is not in the kernel of $H_*(\partial(V\times \D);\Q)\to H_*(V\times \D;\Q)$, then we can find a linear combination of (check) orbits $x=\sum a_i\overline{\alpha}_i$, such that $x$ is closed and $\langle \delta_{\partial}(x), [\Lambda]\rangle\ne 0$, where $\la \cdot,\cdot\ra$ is the pairing $H^*(\partial W;\Q)\otimes H_*(\partial W;\Q)\to \Q$. For simplicity, we assume $x$ is a simple orbit $\overline{\alpha}$ with coefficient $1$; otherwise, the argument in the general case only differs in notation.  As a consequence, for generic choices of almost complex structures, we have 
\begin{enumerate}
	\item $\#\overline{\cM}_H(\overline{\alpha},\overline{\beta})=0$ for any other non-constant Hamiltonian orbit $\overline{\beta}$;
	\item $\#\overline{\cM}_H(\overline{\alpha},\Lambda) \ne 0$,
\end{enumerate}	 
where by counting the number of points, we implicitly require that the virtual dimensions are zero. The Conley-Zehnder index of $\overline{\alpha}$ using the natural bounding disk is $2$.

\subsubsection{Case 2: $Y=\partial W$ for a flexible $W$ with $c_1(W)$ torsion}
By the work of Lazarev \cite{zbMATH07184228}, there exists a contact form $\alpha$ on $\partial W$ and a positive real number $D$, such that all Reeb orbits of period smaller than $D$ have the following properties.
\begin{enumerate}
	\item They are non-degenerate and have Conley-Zehnder indices at least $1$, where the Conley-Zehnder indices are computed using a trivialization of $\det_{\C} \oplus^N\xi$ for some $N\in \N$ as $c_1(W)$ is torsion.
	\item $SH^*_{+,<D}(W)\to H^{*+1}(W;\Q)$ is surjective. 
\end{enumerate}	
We may rescale $\alpha$ such that $D=2$. Let $H$ be a Hamiltonian with slope $2$ that vanishes on the filling. If $[\Lambda]$ is not in the kernel of $H_*(\partial W;\Q)\to H_*(W;\Q)$, we can find a linear combination of Hamiltonian orbits (check and hat orbits from Reeb orbits), which we assume to be a single orbit $\overline{\alpha}$ as before, whose Conley-Zehnder index is $2$, such that $\overline{\alpha}$ represents a closed cochain in the positive cochain complex and $\la \delta_{\partial}(\overline{\alpha}),[\Lambda]\ra \ne 0$. Hence for generic choices of almost complex structures, we also have
\begin{enumerate}
	\item $\#\overline{\cM}_H(\overline{\alpha},\overline{\beta})=0$, for any other non-constant Hamiltonian orbit $\overline{\beta}$;
	\item $\#\overline{\cM}_H(\overline{\alpha},\Lambda) \ne 0$.
\end{enumerate}	 
With the notions above, we consider those curves in the symplectization via neck stretching.
\begin{proposition}\label{prop:curve}
Let $Y$ be $\partial(V\times \D)$ or flexible fillable with the additional assumption that $c_1(Y)$ is torsion. Let $H$ be the Hamiltonian considered above but viewed as defined on $\widehat{Y}$. Then we have the following:
\begin{enumerate}
    \item\label{m2} $\#\overline{\cM}_{Y,H}(\overline{\alpha},\overline{\beta})=0$ for those moduli spaces with expected dimension $0$, where $\overline{\beta}$ is a non-constant Hamiltonian orbit. 
    \item\label{m1} $\overline{\cM}_{Y,H}(\overline{\alpha},\overline{\beta},\Gamma)=\emptyset$ for $\Gamma\ne \emptyset$. 
    \item\label{m3} $\overline{\cM}_{Y,H}(\overline{\alpha},\Lambda,\Gamma)=\emptyset$ for $\Gamma \ne \emptyset$ and $\#\overline{\cM}_{Y,H}(\overline{\alpha},\Lambda) \ne 0$.
\end{enumerate}
\end{proposition}
\begin{proof}
When $Y=\partial(V\times \D)$, this follows from the proof of \cite[Propositions 2.7, 3.2]{zbMATH07673358} by neck-stretching. More precisely, the proof of \cite[Proposition 2.7]{zbMATH07673358} shows that the positive symplectic cohomology of $\partial(V\times \D)$ of period at most $2$ counts curves contained outside of the filling by neck-stretching. In particular, it can be identified with curves in $\widehat{Y}$. Similarly, the proof of \cite[Proposition 3.2]{zbMATH07673358} shows that $\la\delta_{\partial}(\cdot),[\Lambda]\ra:SH^{*}_{+,<2}(V\times \D)\to \Q$ counts curves contained outside of the filling by neck-stretching, hence \eqref{m3} holds. \eqref{m2} follows from simple action considerations, see e.g.\ \cite[(3) of \S 4.3]{bowden2022making}.

For $Y=\partial W$ for flexible $W$ with $c_1(W)$ torsion, \eqref{m2} and \eqref{m3} follow from the proof of \cite[Theorem A]{zbMATH07367119}. Namely, the analogous curves contributing to the positive symplectic cohomology and $\delta_{\partial}$ are contained outside of the filling by neck-stretching and index constraint, e.g.\ \cite[(4) of \S 4.3]{bowden2022making}. For \eqref{m1}, since we can assume all Hamiltonian orbits have Conley-Zehnder indices at least $1$ and SFT degrees of all Reeb orbits of period smaller than $2$ are positive by \cite{zbMATH07184228}, the claim follows from the fact that the expected dimension 
$$\mu_{CZ}(\overline{\alpha})-\mu_{CZ}(\overline{\beta})-1-\sum_{\gamma\in \Gamma}(\mu_{CZ}(\gamma)+n-3)$$
is negative, since $\mu_{CZ}(\overline{\alpha})=2$.
\end{proof}

We consider $\overline{\cM}_{H,W_{\Lambda}}(\overline{\alpha},L,\Gamma)$, the compactified moduli space 
\begin{equation}\label{eqn:WH}
    \overline{\left\{u:\C \backslash \{p_1,\ldots,p_k\} \to \widehat{W_{\Lambda}}\left|
    \begin{array}{c}
    \partial_s u+J(\partial_t u-X_{H})=0, p_i\ne 0,\\
     \displaystyle \lim_{s\to \infty} u(s,\cdot)=\overline{\alpha}, u(0)\in L_{1-\eta}, \lim_{p_i} u= \gamma_i \end{array}\right.\right\}/\R},
\end{equation}
where $\Gamma=\{\gamma_i\}$ is a multiset of good Reeb orbits of $Y_{\Lambda}$, $p_i$ are negative punctures, and $L_{1-\eta}$ is the intersection of the Lagrangian disk $L$ (the co-core of the Weinstein handle in the surgery cobordism) with the complement of $(1-\eta,1)_r\times Y$. $H$ is the Hamiltonian above but viewed as defined on $\widehat{W_{\Lambda}}$.

By stretching the convex boundary of $W_{\Lambda}$, \Cref{prop:corb} implies that $\overline{\cM}_{H,W_{\Lambda}}(\overline{\alpha},\Lambda)$ is not empty with a non-trivial count for a sufficiently stretched almost complex structure. Then the idea to prove \Cref{thm:main} is deforming curves in $\overline{\cM}_{H,W_{\Lambda}}(\overline{\alpha},\Lambda)$ by allowing point to move on the Lagrangian filling $L$ of $\Lambda$. Then the other boundary component shall correspond to the elimination of unit in the contact homology of $Y_{\Lambda}$. More precisely, we have the following. 

\begin{proposition}\label{prop:corbordism}
    We have
    $$\partial_{\CH}\left(\sum\frac{1}{\mu_{\Gamma}\kappa_{\Gamma}}\#\overline{\cM}_{H,W_{\Lambda}}(\overline{\alpha},L,\Gamma)q^{\Gamma}\right)\ne 0 \in \Q,$$hence $\CHA_*(Y_{\Lambda})=0$ and \Cref{thm:main} follows.
\end{proposition}
\begin{proof}
    We consider the boundary configuration of $\overline{\cM}_{H,W_{\Lambda}}(\overline{\alpha},L,\Gamma)$ with virtual dimension $1$. The boundary contains the following cases:
    \begin{enumerate}
        \item\label{t1} Floer breaking. We argue that curves with Floer breaking do not contribute to the boundary. First of all, \eqref{m1} of \Cref{prop:curve} implies that the top-level curve from Floer breaking has no negative punctures asymptotic to Reeb orbits by neck-stretching along $\{1-\eta\}\times Y \subset W_{\Lambda}$. When $Y=\partial W$ for a flexible $W$ with $c_1(W)$ torsion, we have that $c_1(W_{\Lambda})$ is torsion. If a Floer cylinder without negative punctures in $\widehat{W}_{\Lambda}$ has virtual dimension $0$, the Floer cylinder in $\widehat{Y}$ with the same asymptotic orbits also has virtual dimension $0$ as the virtual dimension only depends on $\overline{\beta}$ in this case. Then by neck-stretching along $\{1-\eta\}\times Y \subset W_{\Lambda}$, we see that Floer cylinders with virtual dimension $0$ coincide with \eqref{m2} of \Cref{prop:curve}, as the SFT degrees of orbits are positive. When $Y=\partial(V\times \D)$, we can apply neck-stretching along $\{1-\eta\}\times Y \subset W_{\Lambda}$; then by action reasons, the Floer cylinders in $\widehat{W}_{\Lambda}$ coincide with Floer cylinders in $\widehat{Y}$, which have no negative punctures and the rigid ones are counted as $0$ in \eqref{m2} of \Cref{prop:curve}.
            \begin{figure}[H]
    \centering
    \begin{tikzpicture}
    \draw (0,0) [out=90,in=180] to (0.5,0.25) [out=0,in=90] to (1,0) [out=270,in=0] to (0.5,-0.25) [out=180,in=270] to (0,0);
    \draw node at (0.5,0) {$\overline{\alpha}$};
    \draw node at (0.5,-2) {$\overline{\beta}$};
    \draw node at (0.5,-2.8) {$\overline{\beta}$};
    \draw node at (2.5,-4) {$\gamma$};
    \node at (0.5,-3.5) [circle,fill,inner sep=1.5pt] {};
    \node at (0.5,-3.8) {$\qquad \in L_{1-\eta}$};
    \draw (0,0) to (0,-2) [out=270,in=180] to (0.5,-2.25) [out=0,in=270] to (1,-2) [out=90,in=270]to (1,-1.5) [out=90,in=180] to (1.5,-1.25) [out=0,in=90] to (2,-1.5) [out=270,in=90]to (2,-4) [out=270,in=180] to (2.5,-4.25) [out=0,in=270] to (3,-4) [out=90,in=270] to (3,-1) [out=90,in=0] to (2,-0.5) [out=180,in=270] to (1,0);
    \draw[dashed] (0,-2) [out=90,in=180] to (0.5,-1.75) [out=0,in=90] to (1,-2);
    \draw[dashed] (2,-4) [out=90,in=180] to (2.5,-3.75) [out=0,in=90] to (3,-4);
    \draw (0,-2.8) [out=90,in=180] to (0.5,-2.55) [out=0,in=90] to (1,-2.8) [out=270,in=0] to (0.5,-3.05) [out=180,in=270] to (0,-2.8) [out=270,in=180] to (0.5,-3.5) [out=0,in=270]  to (1,-2.8);
    \end{tikzpicture}
    \qquad
    \begin{tikzpicture}
    \draw node at (0.5,0) {$\overline{\alpha}$};
    \draw node at (2.5,-4) {$\gamma$};
    \node at (0.5,-2.5) [circle,fill,inner sep=1.5pt] {};
    \node at (0.5,-2.8) {$\qquad \in \partial L_{1-\eta}$};
    \draw (0,0) [out=90,in=180] to (0.5,0.25) [out=0,in=90] to (1,0) [out=270,in=0] to (0.5,-0.25) [out=180,in=270] to (0,0);
    \draw  (1,-2) [out=90,in=270]to (1,-1.5) [out=90,in=180] to (1.5,-1.25) [out=0,in=90] to (2,-1.5) [out=270,in=90]to (2,-4) [out=270,in=180] to (2.5,-4.25) [out=0,in=270] to (3,-4) [out=90,in=270] to (3,-1) [out=90,in=0] to (2,-0.5) [out=180,in=270] to (1,0);
    \draw[dashed] (2,-4) [out=90,in=180] to (2.5,-3.75) [out=0,in=90] to (3,-4);
    \draw (0,0) to (0,-2) [out=270,in=180] to (0.5,-2.5) [out=0,in=270] to (1,-2);
    \end{tikzpicture}
    \qquad
    \begin{tikzpicture}
    \draw node at (0.5,0) {$\overline{\alpha}$};
    \draw node at (2.5,-2) {$\gamma$};
    \draw node at (2.5,-2.8) {$\gamma$};
    \draw node at (1.5,-4) {$\gamma_1$};
    \draw node at (3.5,-4) {$\gamma_2$};
    \node at (0.5,-1.75) [circle,fill,inner sep=1.5pt] {};
    \node at (0.5,-2) {$\qquad \in L_{1-\eta}$};
    \draw (0,0) [out=90,in=180] to (0.5,0.25) [out=0,in=90] to (1,0) [out=270,in=0] to (0.5,-0.25) [out=180,in=270] to (0,0);
    \draw  (1,-1.25) [out=90,in=270] to (1,-1) [out=90,in=180] to (1.5,-0.75) [out=0,in=90] to (2,-1) [out=270,in=90]to (2,-2) [out=270,in=180] to (2.5,-2.25) [out=0,in=270] to (3,-2) [out=90,in=270] to (3,-1) [out=90,in=0] to (2,-0.5) [out=180,in=270] to (1,0);
    \draw[dashed] (2,-2) [out=90,in=180] to (2.5,-1.75) [out=0,in=90] to (3,-2);
    \draw (0,0) to (0,-1.25) [out=270,in=180] to (0.5,-1.75) [out=0,in=270] to (1,-1.25);
    \draw (2,-2.8) [out=90,in=180] to (2.5,-2.55) [out=0,in=90] to (3,-2.8) [out=270,in=0] to (2.5,-3.05) [out=180,in=270] to (2,-2.8) [out=270,in=90] to (1,-4) [out=270,in=180] to (1.5,-4.25) [out=0,in=270] to (2,-4) [out=90,in=180] to (2.5,-3.5) [out=0,in=90] to (3,-4) [out=270,in=180] to (3.5,-4.25) [out=0,in=270] to (4,-4) [out=90,in=270] to (3,-2.8);
    \draw[dashed] (1,-4) [out=90,in=180] to (1.5,-3.75) [out=0,in=90] to (2,-4);
    \draw[dashed] (3,-4) [out=90,in=180] to (3.5,-3.75) [out=0,in=90] to (4,-4);
    \end{tikzpicture}
    \caption{The Floer breaking, boundary constraint, and SFT breaking respectively. The first and second curves (with the marked asymptotics) do not appear by \Cref{prop:curve} and neck-stretching}
    \end{figure}
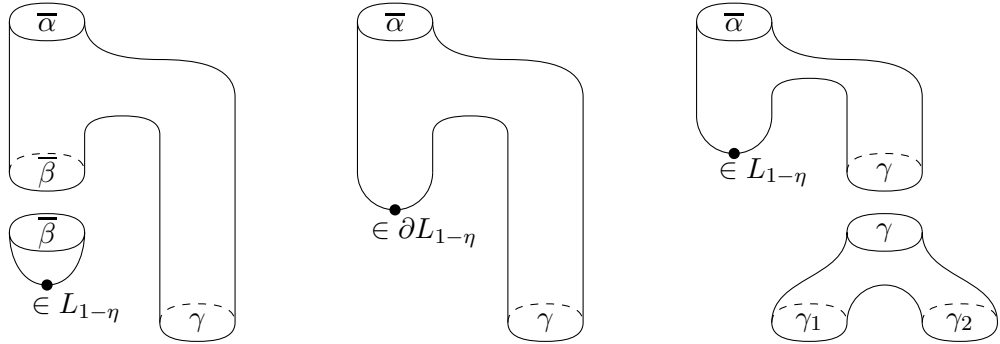
        \item\label{t2} The point constraint goes to the boundary of $L_{1-\eta}$. This can be identified with $\overline{\cM}_{Y,H}(\overline{\alpha},\Lambda)$ by \eqref{m3} of \Cref{prop:curve} using neck-stretching along $\{1-\eta\}\times Y \subset W_{\Lambda}$.
        \item\label{t3} SFT breaking at negative punctures, which corresponds to applying the contact homology differential. 
    \end{enumerate}
    In summary, the fact that 
    $$0=\#\frac{1}{\mu_{\Gamma}\kappa_{\Gamma}}\partial \overline{\cM}_{H,W_{\Lambda}}(\overline{\alpha},L,\Gamma), \text{ where }\virdim \overline{\cM}_{H,W_{\Lambda}}(\overline{\alpha},L,\Gamma)=1$$
    implies that
    $$\partial_{\CH}(\sum_{\virdim = 0}\frac{1}{\mu_{\Gamma}\kappa_{\gamma}}\#\cM_{H,W_{\Lambda}}(\overline{\alpha},L,\Gamma)q^{\Gamma})=\#\overline{\cM}_Y(\overline{\alpha},\Lambda) \ne 0.$$
    To make sense of the above computation, we need to apply Pardon's VFC. Types \eqref{t1} and \eqref{t2} of the boundary, as discussed above, are supplied by curves in \Cref{prop:curve}, which are cut out transversely for a generic choice of almost complex structures that is sufficiently stretched along $\{1-\eta\}\times Y$. Type \eqref{t3} of boundary comes from the SFT breaking at the negative punctures, which is taken care of by the VFC without modifying the count of the curves that are transversely cut out, by \cite[Proposition 4.33]{zbMATH07085531}.
\end{proof}

\subsection{Contact connected sum}\label{ss:42}
In the contact connected sum case, we need to establish the analog of \Cref{prop:curve}. We will use the Viterbo transfer map with respect to augmentations to transfer the curve in $Y_1$ to $Y_1\#Y_2$. 
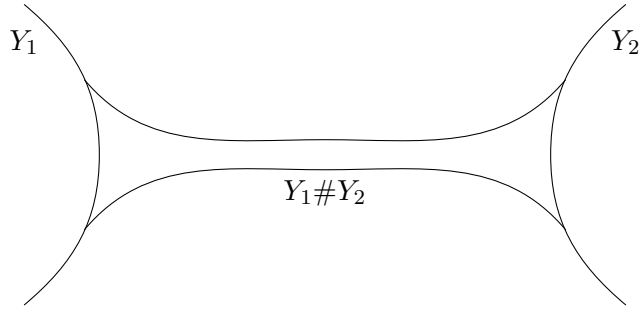
\begin{figure}[H]
    \centering
    \begin{tikzpicture}
    \draw (-4,2) [out=-40,in=90] to (-3,0) [out=-90,in=40] to (-4,-2);
    \draw (4,2) [out=220,in=90] to (3,0) [out=-90,in=140] to (4,-2);
    \draw (-3.2,1) [out=-50,in=180] to (0,0.2) [out=0,in=230] to  (3.2,1);
    \draw (-3.2,-1) [out=50,in=180] to (0,-0.2) [out=0,in=-230] to  (3.2,-1);
    \draw node at (-4,1.5) {$Y_1$};
    \draw node at (4,1.5) {$Y_2$};
    \draw node at (0,-0.5) {$Y_1\# Y_2$};
    \end{tikzpicture}
\end{figure}
We view the above $1$-handle attachment as a degenerate exact cobordism $(W,\lambda)$ from $Y_1\sqcup Y_2$ to $Y_1\# Y_2$, i.e.\ the cobordism has width $0$ on $Y_1\sqcup Y_2$ away from the surgery region. We call this part the thin part of the cobordism. Let $\widehat{W}$ be the completion of $W$ with respect to the Liouville vector field. We use $\overline{\lambda}$ to denote a $1$-form that is the Liouville form $\lambda$ on $W$, and equals $\lambda|_{\partial_+W}$ on $\partial_+ W\times \R_+$ and $\lambda|_{\partial_-W}$ on $\partial_-W\times \R_-$. This form is not smooth along the boundary of the handle, but is smooth everywhere else. We will be using almost complex structures $J$ (possibly depending on other parameters, like the $S^1$-coordinate of the cylinder), such that $\rd \overline{\lambda} (\cdot,J\cdot)\ge 0$ wherever $\rd \overline{\lambda}$ is defined.  A neighborhood of $W$ in $\widehat{W}$ can be colored as follows:
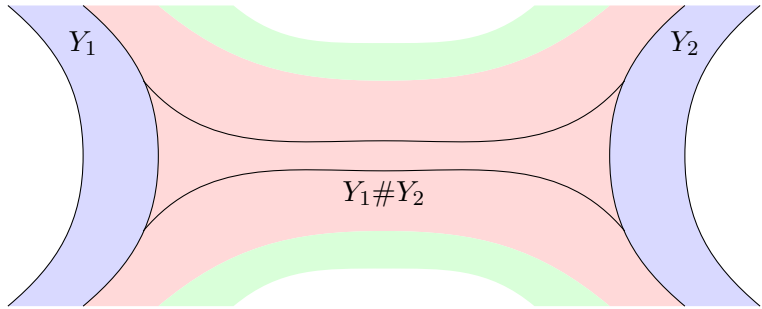
\begin{figure}[H]
    \centering
    \begin{tikzpicture}
    \path [fill=blue!15] (-5,2) to [out=-40, in=90]  (-4,0) to [out=-90,in=40] (-5,-2) to (-4,-2) to [out=40, in=-90] (-3,0) to [out=90,in=-40] (-4,2);
    \path [fill=blue!15] (5,2) to [out=220, in=90]  (4,0) to [out=-90,in=140] (5,-2) to (4,-2) to [out=140, in=-90] (3,0) to [out=90,in=220] (4,2);
    \path [fill=red!15] (-4,2) [out=-40,in=90] to (-3,0) [out=-90,in=40] to (-4,-2) to [out=0,in=180] (-3,-2) to [out=40,in=180] (0,-1) to [out=0,in=140] (3,-2) to [out=0,in=180] (4,-2) to [out=140,in=-90] (3,0) to [out=90,in=220] (4,2) to [out=180,in=0] (3,2) to [out=220,in=0] (0,1) to [out=180,in=-40] (-3,2) to [out=180,in=0] (-4,2);
    \path [fill=green!15]  (-3,-2) to [out=40,in=180] (0,-1) to [out=0,in=140] (3,-2) to [out=180,in=0] (2,-2) to [out=140,in=0] (0,-1.5) to [out=180,in=40] (-2,-2) to [out=180,in=0] (-3,-2);
    \path [fill=green!15]  (-3,2) to [out=-40,in=180] (0,1) to [out=0,in=-140] (3,2) to [out=180,in=0] (2,2) to [out=-140,in=0] (0,1.5) to [out=180,in=-40] (-2,2) to [out=180,in=0] (-3,2);
    \draw (-5,2) [out=-40,in=90] to (-4,0) [out=-90,in=40] to (-5,-2);
    \draw (-4,2) [out=-40,in=90] to (-3,0) [out=-90,in=40] to (-4,-2);
    \draw (4,2) [out=220,in=90] to (3,0) [out=-90,in=140] to (4,-2);
    \draw (5,2) [out=220,in=90] to (4,0) [out=-90,in=140] to (5,-2);
    \draw (-3.2,1) [out=-50,in=180] to (0,0.2) [out=0,in=230] to  (3.2,1);
    \draw (-3.2,-1) [out=50,in=180] to (0,-0.2) [out=0,in=-230] to  (3.2,-1);
    \draw node at (-4,1.5) {$Y_1$};
    \draw node at (4,1.5) {$Y_2$};
    \draw node at (0,-0.5) {$Y_1\# Y_2$};
    \end{tikzpicture}
\end{figure}
Here the concave boundary of the blue region and the convex boundary of the red region are slices of the symplectization of $Y_1\cup Y_2$ and $Y_1\#Y_2$ respectively. The blue region is from flowing $\partial_-W$ along the negative Liouville vector field by time $\ln 2$, namely $([1/2,1]_r\times \partial_-W, r\lambda|_{\partial_-W})$. The red region outside of the handle is from flowing $\partial_+W$ along the Liouville vector field by time $\ln 2$. Hence the convex boundary of the red region is $(\partial_+W, 2\lambda|_{\partial_+W})$. The green region is from flowing along the Liouville vector field by time $\ln 2$, hence is $([2,4]_r\times \partial_+W, r\lambda|_{\partial_+W})$.
We consider an \emph{autonomous} Hamiltonian $H_{VT}$ in the Viterbo transfer for $\widehat{W}$ that is a perturbation to the following
\begin{enumerate}
    \item $H_{VT}$ is zero below the blue region.
    \item $H_{VT}$ is linear on the blue region with slope $4$. Here the slope is defined with respect to the concave boundary of the blue region. In particular, it is of slope $8$ with respect to the $r$-coordinate used in the identification above.
    \item $H_{VT}$ is constant on the red region.
    \item $H_{VT}$ is linear on the green region with slope $4$ (slope $2$ with respect to the $r$-coordinate used in the identification above).
\end{enumerate}
As a consequence, $X_{H_{VT}}$ is zero in the red region and is parallel to the Reeb vector field on $Y_1\sqcup Y_2$ and $Y_1\#Y_2$ in the blue and green regions respectively. As a consequence, we have $\int u^*\rd \overline{\lambda} \ge 0$, whenever $u$ solves the Floer equation for $H_{VT}$ on $\widehat{W}$ possibly with extra negative punctures. Under the conditions of \Cref{thm:main'}, we can assume the contact form on $Y_1$ satisfies the properties in \S \ref{ss:41} and $Y_2$ has a very large contact form such that all Reeb orbits have period $\gg 8$. After applying the thin $1$-handle to $Y_1\sqcup Y_2$, by \cite[Lemma 4.3]{bowden2022making}, Reeb orbits on $Y_1\#Y_2$ of period smaller than $2$ consist of
\begin{enumerate}
    \item Reeb orbits with period smaller than $2$, i.e.\ those considered in  \S \ref{ss:41}. We use $\gamma_{\#}$ to stand for the orbit on $Y_1\#Y_2$ corresponding to $\gamma$ on $Y_1$. 
    \item Multiple covers $\gamma_{h,i}^k$ of the Reeb orbit $\gamma_{h,i}$ for $1\le i \le (\dim Y-1)/2$ contained in a standard contact sphere of dimension $\dim Y-2$ in the belt of the $1$-handle.
    The Conley-Zehnder index of $\gamma_{h,i}^k$ is $(\dim Y-1)/2-1+(k-1)(\dim Y-1)+2i \ge (\dim Y+1)/2$, where the Conley-Zehnder index is computed from bounding discs contained in the handle.
\end{enumerate}
Therefore, $H_{VT}$ only sees those orbits in the green region due to the rescaling. In the blue region, $H_{VT}$ sees Reeb orbits on $(Y_1\cup Y_2,\lambda|_{\partial_-W})$ with period at most $8$, which are always on $Y_1$ by our assumption. 

\subsubsection{The first case: $c_1(Y_1),c_1(Y_2)$ are torsion}\label{sss:421} We can assume that the SFT degrees of orbits with periods up to $8$ are strictly positive; these degrees are globally defined. So even though $\CC(Y_1\cup Y_2,\lambda|_{\partial_-W})$ may not have an augmentation, the part with action at most $8$ does. Specifically, we have the trivial map $\epsilon_0$ sending every generator to zero (by our degree assumptions), or equivalently the augmentation induced from the natural filling of $Y_1$ that is trivial on orbits of $Y_2$ (which are invisible with the action threshold above). We consider a Hamiltonian $H$ on the completion of $(Y_1\#Y_2, 2\lambda|_{\partial_+W})$ with slope $4$ and a Hamiltonian $G$ on the completion of $(Y_1\cup Y_2,\frac{1}{2}\lambda|_{\partial_-W})$ with slope $4$. Using $H_{VT}$ (and the cascades model to deal with the $S^1$ family of Hamiltonian orbits) we have the following Viterbo transfer map
\begin{equation}\label{eqn:diagram}
\xymatrix{
SH_{+,<4}^*(Y_1\#Y_2, 2\lambda|_{\partial_+W},\epsilon_0\circ \phi_{\overline{W}})\ar[d]^{\simeq}  & SH^*_{+,\le 4}(Y_1,\frac{1}{2}\lambda|_{\partial_-W},\epsilon_0)\\
SH_{+,<4}^*(\overline{W},\widehat{\lambda},\epsilon_0) \ar[r]^{\phi_{Viterbo}\qquad }\ar[d] & SH_{+,<4}^*(Y_1\sqcup Y_2, \frac{1}{2} \lambda|_{\partial_-W},  \epsilon_0)\ar[u]^{\simeq}\ar[d]^{\delta_{\partial}}\\
H^{*+1}(W;\Q) \ar[r] & H^{*+1}(Y_1\sqcup Y_2;\Q)
}
\end{equation}
where $\overline{W}$ is the cobordism composed of the blue and red regions. In the Viterbo transfer map, which is a continuation map from $H$ to $H_{VT}$ then to the quotient complex determined by $G$, since $\int u^*\rd \overline{\lambda} \ge 0$ for all relevant Floer cylinders, we know that 
\begin{equation}\label{eqn:vit}
    \phi_{Viterbo}(\overline{\gamma}_{\#})=\overline{\gamma}+\sum c_i\overline{\beta}_i
\end{equation}
where $\overline{\gamma}$ is not from the $1$-handle and the $\beta_i$ have smaller period than $\gamma$ on $Y_1$. This follows from the fact that $\int u^*\rd \overline{\lambda} =0$ implies that $u$ is contained in $\R\times \Ima \gamma$. The leading coefficient comes from the transversely cut out trivial cylinder by our choice of Hamiltonian, i.e.\ it only depends on the coordinate from the Liouville vector field in the thin region where $\gamma$ lies.

\begin{proposition}\label{prop:corb}
 In the case \eqref{thm1} of  \Cref{thm:main'}, we consider a Hamiltonian $H$ on the completion of $(Y_1\#Y_2,\lambda|_{\partial_+W})$ with slope $2$. Using the setup in the discussion above for $Y=Y_1\#Y_2$, there exists $x=\sum a_i\overline{\alpha_i}_\#$, such that 
\begin{enumerate}
    \item\label{c1} $\sum a_i\#\overline{\cM}_{Y,H}(\overline{\alpha_i}_{\#},\overline{\beta})=0$, where $\overline{\beta}$ is a non-constant Hamiltonian orbit; 
    \item\label{c2} $\overline{\cM}_{Y,H}(\overline{\alpha_i}_{\#},\overline{\beta},\Gamma)=\emptyset$ for all $i$, $\overline{\beta}$, and non-empty $\Gamma$;
    \item\label{c3} $\overline{\cM}_{Y,H}(\overline{\alpha_i}_{\#},\Lambda,\Gamma)=\emptyset$ for non-empty $\Gamma$ and $\sum a_i\#\overline{\cM}_{Y,H}(\overline{\alpha_i}_{\#},\Lambda) \ne 0$.
\end{enumerate}
\end{proposition}	
\begin{proof} 
    Let $\overline{\alpha}$ be the orbit in \Cref{prop:curve}; then $\mu_{CZ}(\overline{\alpha})=2$.  We can assume $\dim Y\ge 5$; otherwise $Y=\partial(V\times \D)$ with a surface $V$ with boundary, which will be dealt with in \eqref{thm2} of \Cref{thm:main'}. Since $\mu_{CZ}(\check{\gamma}_{h,i}^k),\mu_{CZ}(\hat{\gamma}_{h,i}^k)\ge \frac{\dim Y+1}{2}\ge 3$, $\phi_{Viterbo}$ is an isomorphism in the $\mu_{CZ}=2$ piece by \eqref{eqn:vit}.  We can take $x$ to be $\phi^{-1}_{Viterbo}(\overline{\alpha})$. As the SFT degrees of orbits (with period up to $2$) on $Y$ are positive, the claims for the moduli spaces $\overline{\cM}_{Y,H}$ are equivalent to the same claims for moduli spaces $\overline{\cM}_{W,H}$ on the cobordism $W$ by neck-stretching; the latter appear in the Viterbo transfer map. That $\overline{\cM}_{W,H}(\overline{\alpha_i}_{\#},\Lambda,\Gamma)=\emptyset$ for $\Gamma \ne \emptyset$ follows from dimension counting as $\mu_{CZ}(\overline{\alpha_i}_\#)=2$ and SFT degrees of orbits in $\Gamma$ are positive. That $\overline{\cM}_{W,H}(\overline{\alpha_i}_{\#},\overline{\beta},\Gamma)=\emptyset$ follows from degree reasons when $Y$ is the boundary of a flexible domain (since $\mu_{CZ}(\overline{\beta})\ge 1$ and SFT degrees of orbits are positive) and by action reasons when $Y=\partial(V\times \D)$.  Since $\mu_{CZ}(\delta_{\epsilon_0}(x))=1$ and $\phi_{Viterbo}$ is again an isomorphism, $x$ represents a closed cochain as $\overline{\alpha}$ does.  Finally, the bottom part of \eqref{eqn:diagram} implies that  $\sum a_i\#\overline{\cM}_{W,H}(\overline{\alpha_i}_{\#},\Lambda) \ne 0$.
\end{proof}

\subsubsection{The second case: $Y_1=\partial(V\times \D)$ for a Weinstein domain $V$ without $c_1$ torsion conditions}In this case, we will use an argument motivated by \cite[Lemma 5.5]{zbMATH06562001} regarding the lower bound of energy from the binding region. Our situation is simpler as we are working with the trivial open book $\partial(V\times \D)$. The contact form used in this paper, i.e.\ from \cite[\S 2,1]{zbMATH07673358}, is obtained by rounding the corner of 
$$\left(\partial(V\times \D) = V\times S^1\cup_{\partial V \times S^1} \partial V \times \D, \lambda_V+\frac{r^2}{2\pi} \rd \theta\right)$$
and then putting a perturbation on the page region $V\times S^1$. The key Reeb orbit $\alpha$ winds around the binding $\partial V \times \{0\} \subset \partial V\times \D$ once. 

Now on the binding region $\partial V \times \D$, the contact structure is given by
$$\xi_{\partial V }\oplus \la \partial_x+\frac{y}{2\pi}R_{\lambda_V},\partial_y-\frac{x}{2\pi}R_{\lambda_V} \ra,$$
where $\xi_{\partial V}\subset T\partial V$ is the contact structure $\ker \lambda_V$, $R_{\lambda_V}$ is the Reeb vector field of $(\partial V, \lambda_V)$, and $x,y$ are coordinates on $\D$. The Reeb vector field on the binding region is $R_{\lambda_V}$. We use the following almost complex structure on $\partial V \times \D \times \R_t$
$$J:\xi_{\partial V} \to \xi_{\partial V} \text{ compatible with } \rd \lambda_V, \quad J(\partial_x+\frac{y}{2\pi}R_{\lambda_V})=\partial_y-\frac{x}{2\pi}R_{\lambda_V}, \quad  J(R_{\lambda_V})=\partial_t.$$
We use $\pi_{\D}$ to denote the projection from $\partial V \times \D \times \R_t$ to $\D$. Then it is straightforward to check that $\pi_{\D}$ is $(J,i)$ holomorphic, where $i$ is the standard complex structure on $\D$. When the Hamiltonian only depends on the $t$-coordinate, the Hamiltonian vector field is parallel to $R_{\lambda_V}$. Then if $u$ solves the Floer equation in $\partial V \times \D \times \R_t$
$$\partial_su+J(\partial_tu-X_H)=0,$$
$\pi_{\D} u$ is a holomorphic map to $\D$. Moreover, we have 
\begin{equation}\label{eqn:area}
    \int u^*\rd\lambda \ge \int (\pi_{\D}u)^*\rd (\frac{r^2}{2\pi}\rd \theta)\ge 0,
\end{equation}
see \cite[the proof of Lemma 4.6]{bowden2022making} for details. 

Following \cite[\S 2.1]{zbMATH07673358}, the contact form on $V\times S^1 = \partial(V\times \D)\backslash \partial V \times \D $ is given by 
$\lambda_V +\frac{1}{2\pi f} \rd \theta$, and on $\partial V\times \D$, the contact form is given by $\lambda_V+\frac{r^2}{2\pi}\rd \theta$. Let $\gamma_p$ be the Reeb orbit corresponding to a critical point $p$; then the period of $\gamma_p$ is $\frac{1}{f(p)}$. Moreover, by \cite[Lemma 4.6]{bowden2022making}, we can apply the surgery outside of the $\partial V\times \D$ region and choose the function $f$ in the perturbation of the contact form such that the following holds:
\begin{equation}\label{eqn:condition}
    \frac{1}{f(p)}-1<\int \gamma_{h,i}^*\lambda
\end{equation}
whenever $p$ is not the minimum point of $f$ and for any $i$. We also require that $f$ is self-indexing in the sense that $f(p)\ge f(q)$ if and only if $\ind(p)\ge \ind (q)$ for critical points $p,q$. In the following, we use almost complex structures extending the above almost complex structure on $\partial V \times \D \times \R_t$.

\begin{proposition}\label{prop:area}
    With the setup above, for $Y=Y_1\#Y_2$ and $\overline{W}$ the connected sum cobordism as in \S \ref{sss:421}, we have the following:
    \begin{enumerate}
        \item\label{a1} $\overline{\cM}_{Y,H}(\overline{\gamma}_p,\overline{\gamma}_{h,i}^k)=\emptyset$, $\overline{\cM}_{\overline{W},H}(\overline{\gamma}_p,\overline{\gamma}_{h,i}^k)=\emptyset$ for any $k$, where $p$ is not the minimum point of $f$.
        \item\label{a2}  The Viterbo transfer map in \eqref{eqn:vit} maps $\overline{\gamma}_{h,i}^k$ to zero.
    \end{enumerate}
\end{proposition}
\begin{proof}
     Since the linking number of $\gamma_p$ around the binding $\partial V\times\{0\}$ is $1$ and the linking number of $\gamma_h^k$ around $\partial V\times \{0\}$ is zero, for any curve $u$ in $\overline{\cM}_{Y,H}(\overline{\gamma}_p,\overline{\gamma}_{h,i}^k,\Gamma)$, we must have 
     $$\int \gamma_p^*\lambda-\int(\gamma_{h,i}^k)^*\lambda\ge \int_{u^{-1}(\partial V \times \D \times \R_t)}u^*\rd \lambda\ge \int_{u^{-1}(\partial V \times \D \times \R_t)}(\pi_{\D}u)^*(\frac{r^2}{2\pi}\rd \theta)=1,$$
     where the second inequality is \eqref{eqn:area} and the last equality follows from the linking number. Hence we arrive at a contradiction with \eqref{eqn:condition}. The same applies to $\overline{W}$ as the connected sum was applied outside of the $\partial V\times \D$ region.

    The intersection number of the curve $u$ in the Viterbo transfer map from $\overline{\gamma}_{h,i}^k$ with $\partial V \times \{0\}$ is negative, as orbits on $Y_1$ have positive linking with $\partial V$. However, this contradicts that $\pi_{\D}u$ is holomorphic, which implies that the intersection number is non-negative.  
\end{proof}

\begin{remark}
    \eqref{eqn:condition} cannot be arranged for the minimum point $p$ of $f$, as $\overline{\cM}_{Y,H}(\overline{\gamma}_p,\overline{\gamma}_{h,1})$ is expected to be non-empty in view of the subcritical surgery formula for symplectic cohomology \cite{zbMATH01798901}. 
\end{remark}

\begin{proposition}\label{prop:corb'}
     \Cref{prop:corb} holds for \eqref{thm2} of \Cref{thm:main'}.
\end{proposition}
\begin{proof}
    Let $\overline{\alpha}$ be the orbit in \Cref{prop:curve} for $Y_1=\partial(V\times \D)$; then $\alpha$ has the minimal period by the self-indexing property. Since $\int u^*\rd \overline{\lambda}\ge 0$ for curves in $\overline{W}$, \eqref{eqn:diagram} still holds for the trivial augmentation, as the differentials and continuation maps do not have action room for augmentations.  Then $\phi_{Viterbo}^{-1}(\overline{\alpha})=\overline{\alpha}_\#$ by \eqref{a2} of \Cref{prop:area} and \eqref{eqn:vit}. The fact that $\overline{\cM}_{Y,H}(\overline{\alpha}_\#,\overline{\beta},\Gamma)=\emptyset$ follows from the minimal period of $\alpha$ on $\partial(V\times \D)$ and \Cref{prop:area}, where $\Gamma$ could be the empty set. $\overline{\cM}_{Y,H}(\overline{\alpha}_\#,\Lambda ,\Gamma)=\emptyset$ for $\Gamma \ne \emptyset$ follows from the same reason. We have $\# \overline{\cM}_{\overline{W},H}(\overline{\alpha}_\#,\Lambda)\ne 0$ from \eqref{eqn:diagram}. In the neck-stretching along $Y$, a similar argument to \Cref{prop:area} rules out the possibility of developing negative punctures asymptotic to $\gamma_{h,i}^k$. Moreover, punctures asymptotic to other Reeb orbits are ruled out because $\alpha$ has the minimal period on $\partial(V\times \D)$. Thus, we have $\# \overline{\cM}_{Y,H}(\overline{\alpha}_\#,\Lambda)\ne 0$. 
\end{proof}

\begin{proof}[Proof of \Cref{thm:main'}]
	It follows from the same argument as \Cref{prop:corbordism} using \Cref{prop:corb,prop:corb'}.
\end{proof}	

\subsection{Functorial explanation}\label{ss:43}
The proof of \Cref{thm:main} has a functorial explanation as follows. Let $H$ be the Hamiltonian on $\widehat{Y}$ which is zero on $(0,1)_r\times Y$ and has slope $a$ on $(1,\infty)_r\times Y$. Then the counting of $\overline{\cM}_{Y,H}(x,y,\Gamma)$ makes $C^{-*}_+(H)\otimes \CC_*(Y)$ into a $\CC_*(Y)$-DGA module, where the differential on $C^{-*}_+(H)\otimes \CC_*(Y)$ is given by 
$$\delta(x\otimes w) = (-1)^{|x|}x\otimes \partial_{\CH}(w)+ \sum_{[\gamma],y}\frac{1}{\mu_{\Gamma}\kappa_{\Gamma}} \# \overline{\cM}_{Y,H}(x,y,\Gamma) y\otimes q^{\Gamma}\cdot w$$
for $x\in C_+^{-*}(H)$ and $w\in \CC_*(Y)$. Then by counting $\overline{\cM}_{Y,H}(x,C,\Gamma)$, we get a $\CC_*(Y)$-DGA module map $C^{-*}_+(H)\otimes\CC_*(Y)\to C^{-*+1}(Y)\otimes \CC_*(Y)$, where $C^{-*+1}(Y)\otimes \CC_*(Y)$ is the trivial $\CC_*(Y)$ DGA module and $C^*(Y)$ is the (Morse) cochain complex of $Y$. Symplectic cohomology with respect to an augmentation $\epsilon$ in \S\ref{ss:aug} is the tensor product $(C_+^{-*}(H)\otimes \CC_*(Y))\otimes_{\CC_*(Y)} \Q$, where $\Q$ is considered as a $\CC_*(Y)$ module using the augmentation $\epsilon$.

Now given an exact cobordism $W$ from $Y_-$ to $Y_+$, by counting $\overline{\cM}_{W,H}(x,C,\Gamma)$ similar to \eqref{eqn:WH}, we get a $\CC_*(Y_+)$-DGA module map $C^{-*}_+(H)\otimes\CC_*(Y_+)\to C^{-*+1}(W)\otimes \CC_*(Y_-)$, where $\CC_*(Y_-)$ is viewed as a $\CC_*(Y_+)$-DGA module by the DGA morphism $\Phi_{W}:\CC_*(Y_+)\to \CC_*(Y_-)$ from the cobordism $W$. Then we have the following diagram of $\CC_*(Y_+)$-DGA module maps, which on homology is commutative by a similar argument to \cite[Proposition 3.2]{zbMATH07367119}. 
$$
\xymatrix{
C_+^{-*}(H)\otimes \CC_*(Y_+)\ar[d] \ar[r] & C^{-*+1}(Y_+)\otimes \CC_*(Y_+)\ar[d]\\
C^{-*+1}(W) \otimes \CC_*(Y_-)\ar[r] & C^{-*+1}(Y_+)\otimes \CC_*(Y_-)}
$$
Then \Cref{prop:curve,prop:corb,prop:corb'} imply that $\overline{\alpha}\otimes 1$ is closed in $C_+^{-*}(H)\otimes \CC_*(Y_+)$ and is mapped to $\delta_{\partial}(\overline{\alpha})\otimes 1\in H_*(C^{-*+1}(Y_+)\otimes \CC_*(Y_-))$. If $H^*(W;\Q)\to H^*(Y_+;\Q)$ does not hit $\delta_{\partial}(\overline{\alpha})$, then we must have $1=0\in H_*(\CC_*(Y_-))$. This is the case for the surgery cobordism in \Cref{thm:main,thm:main'}.

\subsection{Infinite non-loose Legendrians}\label{ss:44}
By the work of Lazarev \cite{zbMATH07305775}, any Weinstein structure on $\C^n$ for $n\ge 3$ can be obtained by attaching a critical handle to a Legendrian sphere in $(\partial(DT^*S^{n-1}\times \D),\xi_{\std})\simeq S^{n-1}\times S^n$, where the embedding of the Legendrian is in the same homotopy class as $S^{n-1}\stackrel{\simeq}{\to} S^{n-1}\times \{p\}\subset S^{n-1}\times S^n$. When the Weinstein structure is standard, the attaching Legendrian is the loose one $\Lambda_{loose}$, i.e.\ the Legendrian lift of the Lagrangian zero section in $DT^*S^{n-1}$. 

A formal Legendrian submanifold of a contact manifold $(Y^{2n-1},\xi)$ consists of an embedding of an $n-1$ dimensional manifold $f:\Lambda\subset Y$ and a family of injective bundle maps $F_s:T\Lambda \to  f^*TY $ for $s\in [0,1]$, such that $F_0=\rd f$ and the image of $F_1$ is a Lagrangian subspace of $f^*\xi$ with respect to the natural conformal symplectic structure, or equivalently (up to homotopy) a totally real subspace for a fixed almost complex structure on $\xi$ compatible with the conformal symplectic structure. Formal Legendrian isotopy is an isotopy of those data. Given two Legendrian embeddings $f_0,f_1:\Lambda \to Y$, if they are smoothly isotopic via $f_s$, then the Legendrian embedding $f_1$ is formally Legendrian isotopic to $(f_0,F_s)$, where $F_s = P_s\rd f_s$ and $P_s$ is the parallel transport from $f_s(p)$ back to $f_0(p)$ using an auxiliary complex connection on $\xi$. Therefore $F_1$ and $\rd f_0$ differ by a section $s$ of the bundle $U(f_0^*\xi)$ of unitary automorphisms. If this section is homotopic to the identity section, then $(f_0,F_s)$, and hence $(f_1,\rd f_1)$, is formally Legendrian isotopic to $(f_0,\rd f_0)$. Now we can consider the situation of Legendrian spheres $f_0,f_1:S^{n-1}\to Y$; then $f_0^*\xi$ is isomorphic to $i^*TT^*S^{n-1}$ for the zero section $i:S^{n-1}\to T^*S^{n-1}$, i.e.\ the complexification $T_{\C}S^{n-1}$ of $TS^{n-1}$. Note that $T_{\C}S^{n-1}\oplus \underline{\C}\simeq \underline{\C}^{n}$, we get an inclusion of $U(T_{\C}S^{n-1})$ into a trivial $U(n)$ bundle. Since $U(n-1)\subset U(n)$ is isomorphic on the $2n-2$ skeleton, whether $s$ is homotopic to the identity in $U(T_{\C}S^{n-1})$ is the same as whether $s$ is homotopic to the identity in the trivial $U(n)$ bundle whenever $n-1< 2n-2$, i.e.\ $n>1$. Finally, such a section is always homotopic to the identity by Bott periodicity as $\pi_{n-1}(U(n))=0$ for $n>1$ odd.

\begin{proposition}
    Let $W_1,W_2$ be two different Weinstein structures on $\C^n$ for $n\ge 3$ that are different from the standard one. Then the corresponding attaching Legendrian spheres are non-loose, not Legendrian isotopic, smoothly isotopic, and formally Legendrian isotopic when $n$ is odd.
\end{proposition}
\begin{proof}
    If the attaching Legendrian is loose, then the resulting Weinstein structure must be the standard Weinstein structure by \cite{zbMATH06054083}. As they are different Weinstein structures, the attaching Legendrian cannot be Legendrian isotopic. Finally, applying the smooth Whitney trick, one sees that the attaching Legendrian is smoothly isotopic to $\Lambda_{loose}$. And by the discussion above, it is formally Legendrian isotopic to $\Lambda_{loose}$ if $n$ is odd. 
\end{proof}

In view of the proposition above, \Cref{prop:exotic} follows from the existence of infinitely many exotic Weinstein structures on $\C^n$ for $n\ge 3$. Such structures were first found by Seidel and Smith \cite{zbMATH02242665} for $n$ even; then McLean found infinitely many for $n$ even \cite{zbMATH05553983}. The most flexible and efficient way of constructing infinitely many such exotic structures for any $n\ge 3$ is due to Abouzaid and Seidel \cite{abouzaid2010altering}. Their exoticity and differences are both illustrated using symplectic cohomology.
\bibliographystyle{plain} 
\bibliography{ref}
\Addresses

\end{document}